\documentclass[a4paper, british]{amsart}

\usepackage{amssymb}
\usepackage{babel}
\usepackage{enumitem}
\usepackage{hyperref}
\usepackage[utf8]{inputenc}
\usepackage{newunicodechar}
\usepackage{mathtools}
\usepackage{varioref}
\usepackage[arrow,curve,matrix]{xy}

\usepackage{colortbl}
\usepackage{graphicx}
\usepackage{tikz}

\usepackage{mathrsfs}

\definecolor{linkred}{rgb}{0.7,0.2,0.2}
\definecolor{linkblue}{rgb}{0,0.2,0.6}

\numberwithin{figure}{section}

\usepackage[hyperpageref]{backref}

\setdescription{labelindent=\parindent, leftmargin=2\parindent}
\setitemize[1]{labelindent=\parindent, leftmargin=2\parindent}
\setenumerate[1]{labelindent=0cm, leftmargin=*, widest=iiii}

\newunicodechar{α}{\ensuremath{\alpha}}
\newunicodechar{β}{\ensuremath{\beta}}
\newunicodechar{χ}{\ensuremath{\chi}}
\newunicodechar{δ}{\ensuremath{\delta}}
\newunicodechar{ε}{\ensuremath{\varepsilon}}
\newunicodechar{Δ}{\ensuremath{\Delta}}
\newunicodechar{η}{\ensuremath{\eta}}
\newunicodechar{γ}{\ensuremath{\gamma}}
\newunicodechar{Γ}{\ensuremath{\Gamma}}
\newunicodechar{ι}{\ensuremath{\iota}}
\newunicodechar{κ}{\ensuremath{\kappa}}
\newunicodechar{λ}{\ensuremath{\lambda}}
\newunicodechar{Λ}{\ensuremath{\Lambda}}
\newunicodechar{μ}{\ensuremath{\mu}}
\newunicodechar{ω}{\ensuremath{\omega}}
\newunicodechar{Ω}{\ensuremath{\Omega}}
\newunicodechar{π}{\ensuremath{\pi}}
\newunicodechar{Π}{\ensuremath{\Pi}}
\newunicodechar{φ}{\ensuremath{\phi}}
\newunicodechar{Φ}{\ensuremath{\Phi}}
\newunicodechar{ψ}{\ensuremath{\psi}}
\newunicodechar{Ψ}{\ensuremath{\Psi}}
\newunicodechar{ρ}{\ensuremath{\rho}}
\newunicodechar{σ}{\ensuremath{\sigma}}
\newunicodechar{Σ}{\ensuremath{\Sigma}}
\newunicodechar{τ}{\ensuremath{\tau}}
\newunicodechar{θ}{\ensuremath{\theta}}
\newunicodechar{Θ}{\ensuremath{\Theta}}

\newunicodechar{𝔹}{\ensuremath{\bB}}
\newunicodechar{ℂ}{\ensuremath{\bC}}
\newunicodechar{ℕ}{\ensuremath{\bN}}
\newunicodechar{ℙ}{\ensuremath{\bP}}
\newunicodechar{ℚ}{\ensuremath{\bQ}}
\newunicodechar{ℝ}{\ensuremath{\bR}}
\newunicodechar{ℤ}{\ensuremath{\bZ}}

\newunicodechar{∇}{\ensuremath{\nabla}}

\newunicodechar{⊕}{\ensuremath{\oplus}}
\newunicodechar{⊗}{\ensuremath{\otimes}}
\newunicodechar{Λ}{\ensuremath{\wedge}}
\newunicodechar{→}{\ensuremath{\to}}
\newunicodechar{⨯}{\ensuremath{\times}}
\newunicodechar{∪}{\ensuremath{\cup}}
\newunicodechar{∩}{\ensuremath{\cap}}
\newunicodechar{⊇}{\ensuremath{\supseteq}}
\newunicodechar{⊃}{\ensuremath{\supset}}
\newunicodechar{⊆}{\ensuremath{\subseteq}}
\newunicodechar{⊂}{\ensuremath{\subset}}
\newunicodechar{≥}{\ensuremath{\geq}}
\newunicodechar{≤}{\ensuremath{\leq}}
\newunicodechar{∈}{\ensuremath{\in}}
\newunicodechar{◦}{\ensuremath{\circ}}
\newunicodechar{°}{\ensuremath{^\circ}}
\newunicodechar{…}{\ifmmode\mathellipsis\else\textellipsis\fi}
\newunicodechar{·}{\ensuremath{\cdot}}
\newunicodechar{∅}{\ensuremath{\emptyset}}

\newunicodechar{¹}{\ensuremath{^1}}
\newunicodechar{²}{\ensuremath{^2}}
\newunicodechar{³}{\ensuremath{^3}}

\DeclareFontFamily{OMS}{rsfs}{\skewchar\font'60}
\DeclareFontShape{OMS}{rsfs}{m}{n}{<-5>rsfs5 <5-7>rsfs7 <7->rsfs10 }{}
\DeclareSymbolFont{rsfs}{OMS}{rsfs}{m}{n}
\DeclareSymbolFontAlphabet{\scr}{rsfs}
\DeclareSymbolFontAlphabet{\scr}{rsfs}

\DeclareMathOperator{\Aut}{Aut}
\DeclareMathOperator{\codim}{codim}

\DeclareMathOperator{\const}{const}

\DeclareMathOperator{\Hom}{Hom}

\DeclareMathOperator{\Image}{Image}

\DeclareMathOperator{\rank}{rank}
\DeclareMathOperator{\Ramification}{Ramification}
\DeclareMathOperator{\red}{red}
\DeclareMathOperator{\reg}{reg}
\DeclareMathOperator{\sat}{sat}

\DeclareMathOperator{\Sym}{Sym}
\DeclareMathOperator{\supp}{supp}

\newcommand{\sE}{\scr{E}}
\newcommand{\sF}{\scr{F}}
\newcommand{\sG}{\scr{G}}
\newcommand{\sH}{\scr{H}}

\newcommand{\sL}{\scr{L}}

\newcommand{\sO}{\scr{O}}

\newcommand{\sQ}{\scr{Q}}

\newcommand{\sT}{\scr{T}}

\newcommand{\bB}{\mathbb{B}}
\newcommand{\bC}{\mathbb{C}}

\newcommand{\bN}{\mathbb{N}}

\newcommand{\bP}{\mathbb{P}}
\newcommand{\bQ}{\mathbb{Q}}
\newcommand{\bR}{\mathbb{R}}

\newcommand{\bZ}{\mathbb{Z}}

\theoremstyle{plain}   
\newtheorem{thm}{Theorem}[section]
 
\newtheorem{conjecture}[thm]{Conjecture}
\newtheorem{cor}[thm]{Corollary}
\newtheorem{defn}[thm]{Definition} 

\newtheorem{lem}[thm]{Lemma}

\newtheorem{prop}[thm]{Proposition}

\theoremstyle{remark}

\newtheorem{claim}[thm]{Claim}
\newtheorem{c-n-d}[thm]{Claim and Definition}
\newtheorem{consequence}[thm]{Consequence}
\newtheorem{construction}[thm]{Construction}

\newtheorem{explanation}[thm]{Explanation}
\newtheorem{notation}[thm]{Notation}
\newtheorem{obs}[thm]{Observation}
\newtheorem{rem}[thm]{Remark}

\newtheorem*{rem-nonumber}{Remark}
\newtheorem{setting}[thm]{Setting}
\newtheorem{warning}[thm]{Warning}

\numberwithin{equation}{thm}

\setlist[enumerate]{label=(\thethm.\arabic*), before={\setcounter{enumi}{\value{equation}}}, after={\setcounter{equation}{\value{enumi}}}}

\newcommand{\wtilde}{\widetilde}
\newcommand{\what}{\widehat}

\newcommand\CounterStep{\addtocounter{thm}{1}\setcounter{equation}{0}}

\newcommand{\factor}[2]{\left. \raise 2pt\hbox{$#1$} \right/\hskip -2pt\raise -2pt\hbox{$#2$}}

\makeatletter
\let\saveqed\qed
\renewcommand\qed{%
   \ifmmode\displaymath@qed
   \else\saveqed
   \fi}
\makeatother
\newcommand{\Publication}[1]{}

\newcommand{\subversionInfo}{}
\newcommand{\svnid}[1]{}
\newcommand{\approvals}[1]{}

\usepackage{mathpazo}

\DeclareMathOperator{\CDiv}{CDiv}
\DeclareMathOperator{\Chow}{Chow}
\DeclareMathOperator{\Cl}{Cl}
\DeclareMathOperator{\Div}{Div}
\DeclareMathOperator{\Branch}{Branch}
\DeclareMathOperator{\Defn}{Defn}

\DeclareMathOperator{\Hilb}{Hilb}
\DeclareMathOperator{\lcm}{lcm}

\DeclareMathOperator{\mult}{mult}
\DeclareMathOperator{\OrbiBranch}{OrbiBranch}
\DeclareMathOperator{\snc}{snc}
\DeclareMathOperator{\Univ}{Univ}

\DeclareMathOperator{\vol}{vol}
\DeclareMathOperator{\WDiv}{WDiv}

\author{Benoît Claudon}

\address{Benoît Claudon, Institut Elie Cartan Nancy, Université de Lorraine, site de Nancy,
  B.P.~70239, 54 506 Vandoeuvre-lés-Nancy, Cedex, France.}
\email{\href{mailto:benoit.claudon@univ-lorraine.fr}{benoit.claudon@univ-lorraine.fr}}
\urladdr{\href{http://iecl.univ-lorraine.fr/~Benoit.Claudon}{http://iecl.univ-lorraine.fr/$\sim$Benoit.Claudon}}

\author{Stefan Kebekus}

\address{Stefan Kebekus, Mathematisches Institut, Albert-Ludwigs-Universität
  Freiburg, Eckerstraße 1, 79104 Freiburg im Breisgau, Germany and University of
  Strasbourg Institute for Advanced Study (USIAS), Strasbourg, France}
\email{\href{mailto:stefan.kebekus@math.uni-freiburg.de}{stefan.kebekus@math.uni-freiburg.de}}
\urladdr{\href{http://home.mathematik.uni-freiburg.de/kebekus}{http://home.mathematik.uni-freiburg.de/kebekus}}
\thanks{Stefan Kebekus gratefully acknowledges support through a joint
  fellowship of the Freiburg Institute of Advanced Studies (FRIAS) and the
  University of Strasbourg Institute for Advanced Study (USIAS).
  Behrouz Taji was partially supported by the DFG-Graduiertenkolleg GK1821
  ``Cohomological Methods in Geometry'' at Freiburg.}

\author{Behrouz Taji}

\address{Behrouz Taji, Mathematisches Institut, Albert-Ludwigs-Universität Freiburg, Eckerstraße 1, 79104 Freiburg im Breisgau, Germany}
\email{\href{mailto:behrouz.taji@math.uni-freiburg.de}{behrouz.taji@math.uni-freiburg.de}}
\urladdr{\href{http://home.mathematik.uni-freiburg.de/taji}{http://home.mathematik.uni-freiburg.de/taji}}



\title{Generic positivity and applications to hyperbolicity of moduli spaces}
\date{\today}

\makeatletter
\hypersetup{
  pdfauthor={\authors},
  pdftitle={\@title},
  pdfsubject={\@subjclass},
  pdfkeywords={\@keywords},
  pdfstartview={Fit},
  pdfpagelayout={TwoColumnRight},
  pdfpagemode={UseOutlines},
  bookmarks,
  colorlinks,
  linkcolor=linkblue,
  citecolor=linkred,
  urlcolor=linkred
}
\makeatother

\begin{document}

\begin{abstract}
The proof of the celebrated Viehweg's hyperbolicity conjecture is a consequence
of two remarkable results: Viehweg and Zuo's existence results for global
pluri-differential forms induced by variation in a family of canonically
polarised manifolds and Campana and Păun's vast generalisation of Miyaoka's
generic semipositivity result for non-uniruled varieties to the context of
pairs.  The aim of this chapter is an exposition of Campana-Păun's generic
semipositivity theorem.
\end{abstract}

\maketitle
\approvals{Behrouz & yes \\ Benoît & yes \\ Stefan & yes}
\tableofcontents

%
% Do not edit the following line.  The text is automatically updated by
% subversion.
%
\svnid{$Id: 01-intro.tex 169 2016-10-18 13:42:18Z kebekus $}

\section{Introduction}
\subversionInfo
\approvals{Behrouz & yes \\ Benoît & yes \\ Stefan & yes}

In 1962 Shafarevich conjectured that a smooth family $f°: X° → Y°$ of complex
projective curves of genus at least equal to $2$, parameterized by $Y°= ℙ^1$,
$ℂ$, $ℂ^*$, or an elliptic curve $E$ is isotrivial, so that is there is no
variation in the algebraic structure of the members of the family.  Equivalently
this conjecture can be expressed as the prediction that the base $Y°$ of any
smooth, non-isotrivial family of projective curves with $g≥ 2$ is of log-general
type.  In other words, we have $κ(Y,K_Y+D)=1$ for any smooth compactification
$(Y,D)$ of $Y°$, with snc boundary divisor $D$.  Shafarevich conjecture was
shown by Parshin and Arakelov.
 
To generalise the Shafarevich conjecture to higher dimensional fibres and
parametrizing spaces, Viehweg considered the hyperbolicity properties of the
moduli stack of canonically-polarised manifolds.  Recall that the moduli functor
$\mathcal M$ of canonically-polarised manifolds with fixed Hilbert polynomial,
is equipped with a natural transformation
\begin{equation*}\label{eq:modulimaps}
  Ψ: \mathcal M (·) → \Hom(·, \mathfrak M),
\end{equation*}
where $\mathfrak M$ denotes the coarse moduli scheme associated with $\mathcal M$.
The scheme $\mathfrak M$ was proved by Viehweg to be quasi-projective,
cf.~\cite{Viehweg95}.  Also recall that a complex analytic space $U$ is said to
be Brody hyperbolic if there are no non-constant holomorphic maps $f: ℂ→ U$.  In
the spirit of this definition, Shafarevich's conjecture is equivalent to the
assertion that the base $Y°$ of non-isotrivial, smooth, projective families of
high genus curves is algebraically Brody hyperbolic in the sense that there are
no non-constant morphisms from $ℂ^*$ to $Y°$.
 
Generalising Shafarevich's conjecture, Viehweg predicted that the moduli stack
of canonically-polarised manifolds is not only algebraically Brody hyperbolic
but that it is Brody hyperbolic.  More precisely, a smooth quasi-projective
variety $Y°$ admitting a generically finite morphism $μ: Y° → \mathfrak M$, must
be Brody hyperbolic.  This conjecture was settled by Viehweg and Zuo
in~\cite{VZ03}.  On the other hand, a long-standing conjecture of Lang predicts
that for a quasi-projective $Y°$, Kobayashi hyperbolicity (which is equivalent
to Brody hyperbolicity for projective varieties) implies that all subvarieties
of $Y°$, including $Y°$, are of log-general type.  In the light of Lang's
problem, Viehweg extended his question on the hyperbolic nature of the moduli
stack of canonically-polarised manifolds to the following conjecture.

\begin{conjecture}[Viehweg's hyperbolicity conjecture]\label{conj:Vhypo}
  Let $Y°$ be a smooth quasi-projective variety admitting a generically finite
  morphism $μ: Y°→ \mathfrak M$.  Then, the smooth compactification $(Y,D)$ of
  $Y°$ is of log-general type.
\end{conjecture}

Viehweg's conjecture has attracted the interest of many algebraic geometers for
a long time.  We refer the reader to the survey \cite{Keb13a} for more details,
including references to earlier results that are not mentioned here for lack of
space.

\subsection{Viehweg's hyperbolicity conjecture according to Viehweg-Zuo and Campana-Păun}
\approvals{Behrouz & yes \\ Benoît & yes \\ Stefan & yes}

A general strategy to prove Conjecture~\ref{conj:Vhypo} consists of two main
steps.  Combining deep results of analytic \cite{Zuo00}, algebraic
\cite{Viehweg83} and Hodge theoretic \cite{Gri84} nature, Viehweg and Zuo
construct in a first step a subsheaf of the sheaf of pluri-log differential
forms of the base whose birational positivity captures the variation\footnote{A
  family $f°: X° → Y°$ of canonically-polarised manifolds is said to have
  \emph{maximal variation} if the moduli map $Ψ(f°):Y° → \mathfrak{M}$ is
  generically finite.} in the family.

\begin{thm}[\protect{Existence of pluri-logarithmic forms in the base, cf.~\cite[Thm.~1.4]{VZ02}}]\label{thm:VZ}
  If the smooth family of canonically-polarised manifolds $f°$ has maximal
  variation, then there exist a positive integer $N$ an invertible subsheaf
  $\sL⊆ \Sym^N\bigl(Ω^1_Y\log (D)\bigr)$ such that $κ(Y,\sL) = \dim Y$.
\end{thm}

Theorem~\ref{thm:VZ} immediately resolves the original conjecture of
Shafarevich.  The goal in the second step is to trace a connection between the
birational positivity (bigness) of $\sL$ in Theorem~\ref{thm:VZ} and that of
$K_Y+D$, thus resolving Conjecture~\ref{conj:Vhypo}.  Working along these lines,
the second author and Kovács established Conjecture~\ref{conj:Vhypo} for moduli
stacks of dimension two and three, \cite{KK08, KK10} and see \cite{Keb13a} for
an overview.  The work relied, among other things, on the log-abundance theorem
for surfaces and threefolds.  In the absence of these methods in higher
dimensions, for instance a complete solution to the abundance problem, Campana
and Păun devised an additional tool, namely a vast generalisation of the famous
generic semipositivity result of Miyaoka to the context of pairs with rational
coefficients.  Here, we state their result in its simplest form and we refer the
reader to Section~\ref{sect:semipositivity} for a general statement.
   
\begin{thm}[\protect{Logarithmic generic semipositivity, cf.~\cite[Thm.~2.1]{CP13}}]\label{thm:logCP13}
  Let $(X,D)$ be a reduced, projective, snc pair.  If $K_X+D$ is
  pseudo-effective, then for every ample divisor $H$ on $X$ and every torsion
  free quotient $\sQ$ of $Ω^1_X(\log D)$ we have
  $$
  c_1(\sQ)·[H]^{n-1}≥ 0.  \qed
  $$
\end{thm}
   
Despite its importance, we found the paper \cite{CP13} rather hard to read.
This chapter is meant to serve as an exposition of Campana and Păun's proof of
Theorem~\ref{thm:logCP13} and its application to resolving
Conjecture~\ref{conj:Vhypo}.

\subsection{Structure of the current chapter}
\approvals{Behrouz & yes \\ Benoît & yes \\ Stefan & yes}

In Section~\ref{sect:defs} we gather some preliminary definitions and notions
that are used throughout this chapter.  In Section~\ref{sect:log-diffs} we
review some of the basics of the theory of orbifolds.  In
Section~\ref{sect:fractional} we delve deeper into some technical details that
will be crucial to the proof of the generic semipositivity result in
Section~\ref{sect:failure}.  In Section~\ref{sect:semipositivity} we state
Theorem~\ref{thm:logCP13} in its full generality.  Section~\ref{sect:hyperbol}
sketches the proof of Conjecture~\ref{conj:Vhypo} using this result.  Part~II is
devoted to the proof of the semipositivity result of Campana and Păun.

\subsection{A note on further results}
\approvals{Behrouz & yes \\ Benoît & yes \\ Stefan & yes}

Constructing degenerate Kähler-Einstein metrics, Campana and Păun have
established a second proof of Theorem~\ref{thm:orbSemPos} that works for Kähler
manifolds, \cite{CP14}.

More recently, they strengthened Theorems~\ref{thm:logCP13} and
\ref{thm:orbSemPos} also in another direction, by proving the pseudo-effectivity
of torsion free quotients, \cite{CP15}.  This latter result is specially
significant for the proof of Viehweg's conjecture, as it makes the LMMP methods
redundant.  For a concise exposition of~\cite{CP15} and its application to
Viehweg's problem we refer the reader to the first author's notes written for
the Bourbaki seminar, \cite{Claudon15}.
  
In a slightly different, but closely related, direction a more general version
of Viehweg's conjecture, that is perhaps closer to the spirit of the original
conjecture of Shafarevich, was formulated by Campana.  In this conjecture
Campana proposed the so-called \emph{special} varieties as higher dimensional
analogues of $ℂ$, $ℂ^*$, $ℙ^1$ and $E$ in Shafarevich conjecture.  We refer the
reader to the original paper of Campana, \cite{Cam04}, for the basic definitions
and background in the theory of special varieties.
  
\begin{conjecture}[The isotriviality conjecture]\label{conj:iso}
  Any smooth family of canonically-polarised manifolds $f°:X° → Y°$
  parametrised by a special quasi-projective variety $Y°$ is isotrivial.
\end{conjecture}

Following the strategy of Campana and Păun and by using the result
of~\cite{MR2976311}, Conjecture~\ref{conj:iso} has been settled
in~\cite{Taji16}.  More recently, in~\cite{PS15}, Popa and Schnell have proved a
vast generalisation of Conjecture~\ref{conj:Vhypo} by extending
Theorem~\ref{thm:VZ} to smooth projective families of varieties of general type.
Where their strategy follows the same two-steps approach discussed above, the
main breakthrough in their result comes from an interesting use of the theory of
Hodge modules to extend some crucial Hodge theoretic tools used in~\cite{VZ03}.

\subsection{Acknowledgements}
\approvals{Behrouz & yes \\ Benoît & yes \\ Stefan & yes}

The authors owe a special thanks to Frédéric Campana and Mihai Păun for many
fruitful discussions.

%
% Do not edit the following line.  The text is automatically updated by
% subversion.
%
\svnid{$Id: 02-definitions.tex 175 2016-10-25 09:36:41Z kebekus $}

\section{Definitions and Notation}
\label{sect:defs}
\subversionInfo
\approvals{Behrouz & yes \\ Benoît & yes \\ Stefan & yes}

In the current section we gather some very basic definitions and concepts needed
for the arguments in the later parts of this chapter.  For the more standard
definitions, we refer to~\cite{Ha77}.  The reader who is familiar with these
preliminaries may wish to skip Subsections~\ref{subsect:vars}
to~\ref{subsect:foliate} and move on Subsection~\ref{subsect:adapted}.  In this
chapter, all varieties are defined over $ℂ$.

\subsection{Varieties, subsets, sheaves and pairs}
\label{subsect:vars}
\approvals{Behrouz & yes \\ Benoît & yes \\ Stefan & yes}

Let us begin by introducing the most basic objects, recurrent throughout this
chapter.

\begin{notation}[Small and big sets]
  Let $X$ be a variety.  A subset $S ⊆ X$ is called \emph{small} if its Zariski
  closure satisfies $\codim_X \overline{S} ≥ 2$.  A subset $U ⊆ X$ is called
  \emph{big} if its complement is small.
\end{notation}

\begin{notation}[Families of curves on projective varieties]\label{not:curves}
  Let $X$ be a projective variety.  A family of curves is a smooth subvariety
  $T ⊆ \Hilb(X)$ whose associated subschemes $(C_t)_{t ∈ T}$ are reduced,
  irreducible and of dimension one.  We say that the family \emph{dominates $X$}
  if $∪_{t ∈ T} C_t$ is dense in $X$.  We say that the family \emph{avoids small
    sets} if, given any small set $S ⊂ X$, there exists a dense open $T° ⊂ T$
  such that $C_t ∩ S = ∅$, for all $t ∈ T°$.
\end{notation}

\begin{defn}[Pair]\label{def:E:1}
  A \emph{pair} $(X,Δ)$ consists of a normal variety $X$ and a $ℚ$-Weil divisor
  $Δ$ on $X$ with coefficients in $[0,1] ∩ ℚ$.  A pair $(X,Δ)$ is called
  \emph{snc} if $X$ is smooth and if the support of $Δ$ has simple normal
  crossings only.  We denote the maximal open subset of $X$ where $(X,Δ)$ is
  smooth by $(X,Δ)_{\snc}$.  Note that this is a big subset of $X$.  The
  fractional part of $Δ$ is written as $\{Δ\}$.
\end{defn}

\begin{notation}[Reflexive hull]
  Given a normal, quasi-projective variety $X$ and a coherent sheaf $\sE$ on
  $X$, write
  $$
  Ω^{[p]}_X := \bigl(Ω^p_X \bigr)^{**}, \quad \sE^{[m]} := \bigl(\sE^{⊗ m}
  \bigr)^{**} \quad\text{and}\quad \det \sE := \bigl( Λ^{\rank \sE} \sE
  \bigr)^{**}.
  $$
  Given any morphism $f : Y → X$, write $f^{[*]} \sE := (f^* \sE)^{**}$, etc.
\end{notation}

\subsection{Morphisms}
\approvals{Behrouz & yes \\ Benoît & yes \\ Stefan & yes}

Let us now briefly review some basic notions and properties of morphisms
between normal varieties.

\begin{construction}[Push-forward of Weil divisor]\label{cons:push}
  Let $f : X → Z$ be a morphism of normal varieties.  Recall from ``Zariski's
  Main Theorem in the form of Grothendieck''\footnote{See \cite{EGA4-3}
    Zariski's main theorem in the form of Grothendieck and
    \cite[Thm.~3.8]{GKP13} for the precise statement used here.  A full proof is
    found in the extended version of \cite{GKP13}, available on the arXiv.} that
  there exists a unique, normal variety $\overline{X}$ and a unique
  factorisation $f = β ◦ α$ as follows,
  $$
  \xymatrix{ %
    X \ar[rrr]^{α}_{\text{open immersion}} &&& \overline{X}
    \ar[rrr]^{β}_{\text{proper morphism}} &&& Z.  }
  $$
  Taking Zariski-closures yields a push-forward morphism $α_* : \WDiv(X) →
  \WDiv(\overline{X})$.  Composing with the standard push-forward morphism of
  the proper morphism $β$, cf.~\cite[I~Sect.~1.4]{Fulton98}, we obtain a map
  $f_* : \WDiv(X) → \WDiv(Z)$.
\end{construction}

\begin{rem}[Push-forward vs.\ linear equivalence]\label{rem:push}
  The map $f_*$ of Construction~\ref{cons:push} will in general not respect
  linear equivalence, unless one of the following folds.
  \begin{enumerate}
  \item The morphism $f$ is proper.
  \item\label{il:zzr} The variety $X$ is a big open subset of $Z$ and $f$ is the
    inclusion.  In this case, $f_*$ is an isomorphism.
  \end{enumerate}
\end{rem}

\subsubsection{Galois covers}
\approvals{Behrouz & yes \\ Benoît & yes \\ Stefan & yes}

As we will be working with $G$-sheaves, Galois morphisms are of particular
interest.

\begin{defn}[Cover and covering morphism]\label{def:E:cover}
  Finite, surjective morphisms $γ: Y → X$ between normal varieties $X$, $Y$ are
  called \emph{covers} or \emph{covering morphisms}.
\end{defn}

\begin{defn}[Galois morphism]\label{def:Galois}
  A covering map $γ : X → Y$ of normal varieties is called \emph{Galois} if
  there exists a finite group $G ⊂ \Aut(X)$ such that $γ$ is isomorphic to the
  quotient map.
\end{defn}

\begin{warning}[Galois $\not \Rightarrow$ étale]
  Definition~\ref{def:Galois} does not require $γ$ to be étale.  This will be of
  crucial importance for nearly everything that follows.
\end{warning}

\subsection{Equidimensional morphisms}
\approvals{Behrouz & yes \\ Benoît & yes \\ Stefan & yes}

In a normal variety, Weil divisors need not be Cartier.  Still, it is possible
to define a pull-back map, at least for equidimensional morphisms.

\begin{defn}[Equidimensional morphism]\label{def:equidim}
  Let $f : X → Z$ be a dominant morphism of varieties.  We say that $f$ is
  \emph{equidimensional} if there exists a number $d$ such for any $x ∈ X$,
  the associated fibre $f^{-1}f(x)$ is of pure dimension $d$.  The number $d$ is
  called \emph{relative dimension}.
\end{defn}

\begin{rem}[Preimages of big and small sets]\label{rem:piob}
  In the setting of Definition~\ref{def:equidim}, if $Z' ⊆ Z$ is any algebraic
  set, then $\codim_X f^{-1}(Z') ≥ \codim_Z Z'$.  In particular, preimages of
  small sets are small, and preimages of big sets are big.
\end{rem}

\begin{lem}[Equidimensional morphism and normalisation]
  Let $f : X → Z$ be a dominant, equidimensional morphism of varieties.  If $X$
  is normal and $Z'$ the normalisation of $Z$, then the natural morphism $f' : X
  → Z'$ is likewise equidimensional.
\end{lem}
\begin{proof}
  Recalling that the normalisation morphism $η : Z' → Z$ is finite, it follows
  that for any $x ∈ X$, the fibre $F' := (f')^{-1}f'(x)$ is a union of connected
  components of $F = f^{-1}f(x)$.  If $F$ is of pure dimension $d$, then so is
  $F'$.
\end{proof}

\subsubsection{Pull-back}
\approvals{Behrouz & yes \\ Benoît & yes \\ Stefan & yes}

We now explain the construction of pull-back maps for Weil divisors
in a normal variety.

\begin{construction}[Pull-back of Weil divisor]\label{cons:pull}
  Let $f : X → Z$ be an equidimensional morphism between normal varieties.  We
  define a pull-back morphism $f^* : \WDiv(Z) → \WDiv(X)$ as the composition of
  the following morphisms,
  \begin{multline*}
    \WDiv(Z) \xrightarrow[\text{by \ref{il:zzr} since $Z_{\reg}$ is big}]{\cong}
    \WDiv(Z_{\reg}) \xrightarrow{\cong} \CDiv( Z_{\reg})
    \xrightarrow{\bigl(f|_{f^{-1}(Z_{\reg})}\bigr)^*} \\
    \CDiv \bigl( f^{-1}(Z_{\reg}) \bigr) → \WDiv \bigl( f^{-1}(Z_{\reg}) \bigr)
    \xrightarrow[\text{by \ref{il:zzr} and Rem.~\ref{rem:piob}}]{\cong}
    \WDiv(X).
  \end{multline*}
  Since all morphisms respect linear equivalence, so does $f^*$.
\end{construction}

\subsubsection{Multiplicities, ramification and branch divisors}
\approvals{Behrouz & yes \\ Benoît & yes \\ Stefan & yes}

We briefly review various notions of multiplicities that appear in the following
sections.

\begin{defn}[Multiplicities]\label{defn:mult}
  Let $f : X → Z$ be an equidimensional morphism between normal varieties.  If
  $Δ ∈ \WDiv(X)$ is prime, define the \emph{multiplicity of $f$ along $D$} as
  $$
  \mult_Δ f := \mult_Δ f^* D, \quad \text{where} \quad D := (f_* Δ)_{\red}.
  $$
\end{defn}

\begin{rem}
  In Definition~\ref{defn:mult}, either $\supp Δ$ dominates $Z$, or
  $f_* Δ \ne 0$ and $\mult_Δ f ≥ 1$.
\end{rem}

\begin{defn}[Ramification and branch divisors]\label{def:ramificationandbranch}
  Let $f : X → Z$ be an equidimensional morphism between normal varieties.  The
  \emph{ramification} and \emph{branch divisor} of $f$ are defined as follows,
  \begin{align*}
    \Ramification f & := \sum_{\overset{Δ ∈ \WDiv(X)}{\text{prime}}} \max \bigl\{ 0, \mult_Δ f - 1 \bigr\} · Δ \\
    \Branch f & := \sum_{\qquad\mathclap{\overset{D ∈ \WDiv(Z)\text{ prime}}{D ≤ f_* \Ramification f}}\qquad} \lcm \{ \mult_Δ f^* D \,|\, Δ ⊆ \supp f^* D \text{ a prime div.} \} · D \\
    \OrbiBranch f & := \sum_{\qquad\mathclap{\overset{D ∈ \WDiv(Z)\text{ prime}}{D ≤ \Branch f}}\qquad} \frac{1- \mult_D \Branch f}{\mult_D \Branch f} · D
  \end{align*}
\end{defn}

\begin{rem}\label{rem:brram}
  In the setting of Definition~\ref{def:ramificationandbranch}, observe that
  $$
  f \bigl( \supp \Ramification f \bigr) ⊆ \supp \Branch f.
  $$
  If $f$ is proper, then equality holds.
\end{rem}

\subsubsection{Local normal form}
\approvals{Behrouz & yes \\ Benoît & yes \\ Stefan & yes}

In this section we give explicit description of equidimensional morphisms in a
suitably chosen local analytic coordinate system.

\begin{construction}[Local normal form]\label{cons:E:8}
  Let $f: X → Z$ be an equidimensional morphism of normal varieties, of relative
  dimension $d$.  Let $Z° ⊆ Z$ be the largest open set such that both $Z°$ and
  $Z° ∩ (\supp \Branch f)$ are smooth.  Let $X° ⊆ f^{-1}(Z°)$ be the largest
  open set such that both $X°$ and $X° ∩ (\supp \Ramification f)$ are smooth,
  and such that the following restrictions of $f$ are smooth morphisms,
  \begin{align*}
    X° \setminus (\supp \Ramification f) & → Z°\\
    X° ∩ (\supp \Ramification f) & → Z° ∩ (\supp \Branch f).
  \end{align*}
  Remark~\ref{rem:brram} ensures that the second map is defined.  Observe that
  both $Z°$ and $X°$ are big open subsets of $Z$ and $X$, respectively.  Let
  $\vec x ∈ X°$ be any point and $\vec z := f(\vec x) ∈ Z°$ be its image.

  If $\vec x$ is \emph{not} contained in the support of $\Ramification(f)$, then
  $f$ is smooth at $\vec x$.  If $z_0, …, z_n ∈ \sO_{Z, \vec z}$ are local
  holomorphic coordinates on $Z$ centred about $\vec z$, then
  $x_i := z_i ◦ f ∈ \sO_{X,\vec x}$ can be completed to a system of holomorphic
  coordinates on $X$ centred about $\vec x$.  In these coordinates, $f$ takes
  the form
  \begin{equation}\label{eq:lnf1}
    f : (x_0, …, x_n, x_{n+1}, …, x_{n+d}) \mapsto (x_0, …, x_n).
  \end{equation}

  If $\vec x$ \emph{is} contained in the support of $\Ramification(f)$, then
  there exists a holomorphic function $z_0 ∈ \sO_{Z,\vec z}$ which locally
  generates the ideal of the smooth hypersurface
  $\bigl(\supp \Branch (f)\bigr)$.  Near $\vec x$, there exists a holomorphic
  function $x_0 ∈ \sO_{X°,\vec x}$ such that $z_0 ◦ f = x_0^m$, where $m$ is the
  order of ramification of $f$ along the unique component of $\Ramification(f)$
  that contains $\vec x$.  Completing, we obtain holomorphic coordinate
  functions of the following form,
  $$
  z_0, …, z_n ∈ \sO_{Z, \vec z} \quad \text{and} \quad x_0,
  \underbrace{x_1}_{\mathclap{:= z_1 ◦ f}}, …, \underbrace{x_n}_{\mathclap{:=
      z_n ◦ f}}, x_{n+1}, …, x_{n+d} ∈ \sO_{X, \vec x}.
  $$
  In these coordinates, $f$ takes the form
  \begin{equation}\label{eq:lnf2}
    f : (x_0, x_1, …, x_n, x_{n+1}, …, x_{n+d}) \mapsto (x_0^m, x_1, …, x_n).
  \end{equation}
\end{construction}

\begin{notation}[Local normal form]\label{not:localNormalForm}
  In the setting of Construction~\ref{cons:E:8}, we refer to the explicit
  description of $f$ in \eqref{eq:lnf1} and \eqref{eq:lnf2} as \emph{local
    normal forms}.  We call $X°$ and $Z°$ the \emph{maximal open sets where $f$
    can locally be written in normal form}.  If $X° = X$, we say that \emph{$f$
    can locally be written in normal form}.
\end{notation}

\subsection{Rational maps}
\approvals{Behrouz & yes \\ Benoît & yes \\ Stefan & yes}

A certain class of rational maps and divisors appear naturally in the proof of
the semipositivity theorem, Theorem~\ref{thm:orbSemPos}.  The current section is
devoted to introducing these maps and reviewing some of their basic properties.

\begin{notation}[Domain of definition, preimages, connected fibres]
  Let $f : X \dasharrow Z$ be a rational map between varieties.  We denote the
  domain of definition of $\Defn(f) ⊆ X$ and write $f_{\Defn}$ for the morphism
  $f|_{\Defn(f)}$.  Given any subset $Z' ⊆ Z$, write $f^{-1}(Z')$ as a shorthand
  for $f_{\Defn}^{-1} (Z')$.  If $z ∈ Z$ is any point, call $X_z := f^{-1}(z)$
  the \emph{fibre over $z$}.  We say that $f$ \emph{has connected fibres} if
  $f_{\Defn}$ has connected fibres.
\end{notation}

\begin{notation}[Horizontal and vertical divisors]\label{not:decomphv}
  Let $f : X \dasharrow Z$ be a rational map between normal varieties.  If $Δ$
  is any prime divisor on $X$, observe that $Δ$ intersects $\Defn(f)$
  non-trivially.  Call $Δ$ \emph{horizontal} if $Δ∩\Defn(f)$ dominates $Z$,
  otherwise call it \emph{vertical}.  If $Δ$ is any $ℚ$-Weil divisor, there
  exists an associated decomposition $Δ = Δ^{horiz} + Δ^{vert}$.
\end{notation}

\begin{defn}\label{defn:esseqdim}
  Let $f : X \dasharrow Z$ be a rational map between normal varieties.  We say
  that $f$ is \emph{essentially equidimensional} if there exists an open set $U
  ⊆ \Defn(f)$ that is big in $X$ such that $f|_U$ is an equidimensional
  morphism.
\end{defn}

\begin{rem}
  Recall from Definition~\ref{def:equidim} that equidimensional morphisms (and
  hence essentially equidimensional maps) are dominant.
\end{rem}

\begin{defn}[Ramification and branch divisors for essentially equidimensional maps]\label{defn:RessEq}
  Let $f : X \dasharrow Z$ be an essentially equidimensional rational map
  between normal varieties as in Definition~\ref{defn:esseqdim}.  Define
  $$
  \Ramification f := \Ramification \bigl( f_{\Defn}|_U \bigr).
  $$
  Ditto for $\Branch f$ and $\OrbiBranch f$.
\end{defn}

\subsubsection{Pull-back and push-forward}
\approvals{Behrouz & yes \\ Benoît & yes \\ Stefan & yes}

In the rest of the current subsection we focus on the behaviour of cycles and
sheaves under $f_*$ and $f^*$, given that $f$ is a rational map between normal
varieties.

\begin{construction}[Push-forward for rational map]
  Let $f : X \dasharrow Z$ be a rational map between normal varieties.  Since
  $\Defn(f)$ is a big subset of $X$, we obtain a canonical identification
  $\WDiv \bigl( \Defn(f) \bigr) \cong \WDiv(X)$.  Construction~\ref{cons:push}
  therefore gives a push-forward map $f_* : \WDiv(X) → \WDiv(Z)$.
\end{construction}

\begin{construction}[Pull-back of sheaves]\label{cons:pbs}
  Let $f : X \dasharrow Z$ be a rational map between normal varieties.  If $\sF$
  is any coherent sheaf of $\sO_Z$-modules, write
  $$
  f^{[*]} \sF := \bigl(ι_* f_{\Defn}^* \sF \bigr)^{**},
  $$
  where $ι : \Defn(f) → X$ is the inclusion.  Since $\Defn(f)$ is a big, open
  set, this is a coherent, reflexive sheaf on $X$.
\end{construction}

\begin{construction}[Pull-back of divisors for essentially equidimensional map]\label{cons:pbwd}
  Let $f : X \dasharrow Z$ be an essentially equidimensional rational map.
  Item~\ref{il:zzr} of Remark~\ref{rem:push} gives a canonical identification
  $\WDiv(U) \cong \WDiv(X)$ that respects linear equivalence.
  Construction~\ref{cons:pull} therefore gives a pull-back map
  $f^* : \WDiv(Z) → \WDiv(X)$, which does not depend on the choice of $U$,
  respects linear equivalence, and therefore induces a morphism between divisor
  class groups.
\end{construction}

\begin{rem}[Pull-back for Weil divisorial sheaves]
  The pull-back Constructions~\ref{cons:pbs} and \ref{cons:pbwd} are compatible
  for Weil divisorial sheaves.  More precisely, if $f : X \dasharrow Z$ is any
  essentially equidimensional rational map between normal varieties and if
  $D ∈ \WDiv(Z)$, then $f^{[*]} \sO_Z(D) \cong \sO_X\bigl( f^* D \bigr)$.
\end{rem}

\subsubsection{Relative tangent sheaves}
\label{subsubsect:relative}
\approvals{Behrouz & yes \\ Benoît & yes \\ Stefan & yes}

The aim of this section is to establish an explicit description for the relative
canonical sheaf of an essentially equidimensional rational map.

\begin{construction}[Relative tangent sheaf]\label{constr:reltgt}
  Let $f : X \dasharrow Z$ be an essentially equidimensional rational map
  between normal varieties.  Recall from Remark~\ref{rem:piob} that there exists
  a big, open set $U ⊆ f^{-1}(Z_{\reg}) ∩ X_{\reg}$ such that
  $f|_U : U → Z_{\reg}$ is an equidimensional morphism.  Denote the inclusion by
  $ι: U → X$, consider the kernel
  $$
  \sT_{U/Z_{\reg}} := \ker \bigl( \sT_U → (f|_U)^* \sT_{Z_{\reg}}),
  $$
  and set $\sT_{X/Z} := ι_* \sT_{U/Z_{\reg}}$.  By construction, $\sT_{X/Z}$ is
  a reflexive subsheaf of $\sT_X$, and in fact a foliation (see
  Notation~\ref{not:foliation} below).  The sheaf $\sT_{X/Z}$ is independent of
  the choice of $U$.
\end{construction}

\begin{construction}[Relative canonical class]
  Let $f : X \dasharrow Z$ be an essentially equidimensional rational map.
  Construction~\ref{cons:pbwd} allows to define the relative canonical class as
  $[K_{X/Z}] := [K_X] - f^* [K_Z] ∈ \Cl(X)$.
\end{construction}

\begin{lem}[Determinant of relative tangent sheaf]\label{lem:relTx}
  In the setting of Construction~\ref{constr:reltgt},
  $$
  \det \sT_{X/Z} \cong \sO_X \bigl( - K_{X/Z} + \Ramification f \bigr).
  $$
\end{lem}
\begin{proof}
  Since both sides of the equation are reflexive sheaves, it suffices to show
  equality on the big, open subset of $U$ where the morphism $f|_U$ can locally
  be written in normal form.  There, the claim follows from an elementary
  computation in coordinates.
\end{proof}

\subsection{Foliations}
\label{subsect:foliate}
\approvals{Behrouz & yes \\ Benoît & yes \\ Stefan & yes}

The notion of a foliation being transversal to a divisor is a recurrent theme in
this chapter.  Let us briefly spell out what is meant by this.

\begin{notation}[Foliation]\label{not:foliation}
  Let $X$ be a normal variety.  A \emph{foliation} is a saturated subsheaf
  $\sT_X$, whose restriction to $X_{\reg}$ is closed under the Lie-bracket.
\end{notation}

\begin{rem}
  In the setting of Notation~\ref{not:foliation}, recall that the tangent sheaf
  $\sT_X := \bigl( Ω^1_X \bigr)^{*}$ is reflexive.  As a saturated subsheaf, the
  foliation $\sF$ is likewise reflexive.
\end{rem}

\begin{notation}\label{not:ONeil}
  For any saturated subsheaf $\sF$ of $\sT_X$, the Lie-bracket, which is defined
  only on $X_{\reg}$, induces an $\sO_{X}$-linear map
  $$
  N: \sF^{[2]} → (\sT_X/\sF)^{**}
  $$
  known as the O'Neil tensor.  Vanishing of $N$ characterises then $\sF$ as
  being a foliation.
\end{notation}

\begin{notation}[Divisors generically transveral to a foliation]\label{not:genTrans}
  In the setting of Notation~\ref{not:foliation}, there exists a big open set
  $U ⊆ X$, contained in $X_{\reg}$ where $\sF$ is a subvectorbundle of $\sT_X$.
  If $D ⊂ X$ is any prime divisor, then $\sF$ is as subvectorbundle of $\sT_X$
  near general points $x ∈ D$.  In particular, one can check whether $\sF$ is
  transversal to $D$ at $x$.  This allows to decompose any $ℚ$-Weil divisor $Δ$
  as $Δ = Δ^{trans} + Δ^{ntrans}$, where $Δ^{trans}$ consists of those
  components that are generically transversal to $\sF$.
\end{notation}

\begin{rem}
  In the setting of Construction~\ref{constr:reltgt}, the decompositions of
  $ℚ$-Weil divisors given in Notations~\ref{not:decomphv} and
  \ref{not:genTrans} agree.  More precisely, if $Δ$ is any $ℚ$-Weil divisor on
  $X$, then $Δ^{horiz} = Δ^{trans}$ and $Δ^{vert} = Δ^{ntrans}$.
\end{rem}

\subsection{Adapted morphisms}
\label{subsect:adapted}
\approvals{Behrouz & yes \\ Benoît & yes \\ Stefan & yes}

In this section we introduce a class of morphisms, called ``adapted'', that is
indispensable to even formulate our main result.

\begin{defn}[Adapted and strongly adapted cover cover]\label{def:E:5}
  Let $(X, Δ)$ be a pair and decompose $Δ$ into irreducible components,
  $Δ = \sum δ_j D_j$.  Let $γ: Y → X$ be an essentially equidimensional morphism
  to a normal variety $Y$.  The morphism $γ$ is called \emph{adapted to $(X,Δ)$}
  if $γ^* Δ$ is an integral divisor.  The morphism $γ$ is called \emph{strongly
    adapted to $(X,Δ)$} if for any index $j$ with $δ_j ∈ (0,1)$, written as
  $δ_j = \frac{a_j}{b_j}$ with $a_j$, $b_j ∈ ℕ$ coprime, and any Weil divisor
  $E \subsetneq Y$, we have
  $$
  \mult_E (γ^*D_j) ∈ \{ 0, b_j\}.
  $$
  In other words, if $E$ appears in $γ^*D_j$ at all, then its multiplicity must
  be $b_j$ precisely.
\end{defn}

\subsubsection{Existence}
\approvals{Behrouz & yes \\ Benoît & yes \\ Stefan & yes}

The existence of adapted morphisms, when $(X,Δ)$ is an snc pair, was established
by Kawamata, cf.~\cite[Prop.~4.1.12]{Laz04-I}.  Here we briefly review the case
where $(X,Δ)$ is \emph{not} snc.  See~\cite[Subsect.~1.2]{CP13} for an
alternative construction.

\begin{prop}[Existence of strongly adapted, Galois covers]\label{prop:exsam}
  Let $(X, Δ)$ be a pair.  Then, there exists a cyclic Galois cover $γ: Y → X$
  that is strongly adapted to $(X,Δ)$.
\end{prop}
\begin{proof}
  Let $π : \wtilde X → X$ be a log resolution of singularities, and consider the
  $ℚ$-Weil divisor given as the strict transform, $\wtilde D := γ_*^{-1}Δ$.  The
  existence of cyclic Galois cover $\wtilde{γ} : \wtilde Y → \wtilde X$ that is
  strongly adapted to $(\wtilde X, \wtilde D)$ has been recalled in
  \cite[Prop.~2.9]{MR2860268}.  We obtain the desired covering by Stein
  factorisation of the composed morphism $\wtilde Y → X$,
  $$
  \xymatrix{ %
    \wtilde Y \ar[rr]_{α} \ar@/^5mm/[rrrr]^{\wtilde{γ}◦ π } && \wtilde X \ar[rr]_{β} && X.
  }
  $$
  The equivariant version of Zariski's Main Theorem, \cite[Theorem~A.1]{GKP13},
  guarantees that $β$ is Galois, and that its Galois group equals that of
  $\wtilde{γ}$.
\end{proof}

\subsubsection{Relation to earlier definitions}
\approvals{Behrouz & yes \\ Benoît & yes \\ Stefan & yes}

Definition~\ref{def:E:5} is equivalent to various other definitions of adapted
morphisms that appear in the literature ---it goes without saying that all are
various takes on the original definition of Campana.  To see this, it is
convenient to first introduce the following definition of the round-up of a
$ℚ$-divisor.

\begin{defn}[$\mathcal C$-round-up]\label{def:E:4}
  Let $(X, Δ)$ be a pair and decompose $Δ$ into irreducible components,
  $Δ = \sum δ_j D_j$.  If $j$ is any index with $δ_j ∈ (0,1)$, write
  $δ_j = \frac{a_j}{b_j}$ with $a_j$, $b_j ∈ ℕ$ coprime.  Finally, set
  $$
  \lceil δ_j \rceil_{\mathcal C} =
  \begin{cases}
    0 & \text{if $δ_j = 0$} \\
    1 & \text{if $δ_j = 1$} \\
    \frac{b_j-1}{b_j} & \text{otherwise.}
  \end{cases}
  $$
  We call the divisor $\lceil Δ \rceil_{\mathcal C} := \sum_j \lceil δ_j
  \rceil_{\mathcal C} · D_j$ the \emph{$\mathcal C$-round-up} of $Δ$.  If $Δ
  = \lceil Δ \rceil_{\mathcal C}$, we call $(X,Δ)$ a $\mathcal C$-pair.
\end{defn}

\begin{rem}[Comparison with earlier definitions]
  Let $(X,Δ)$ be a pair and $γ: Y → X$ a covering map.  Then the following are
  equivalent.
  \begin{enumerate}
  \item The cover $γ$ is adapted to $(X,Δ)$ in the sense of
    Definition~\ref{def:E:5}.
  \item The cover $γ$ is adapted to $(X, \lceil Δ \rceil_{\mathcal C})$ in the
    sense of Definition~\ref{def:E:5}.
  \item The morphism $γ$ is adapted to $(X, \lceil Δ \rceil_{\mathcal C})$ in
    the sense of \cite[Definition~2.7]{MR2860268}.
  \end{enumerate}
  Ditto for strongly adapted covers.
\end{rem}

\subsection{Numerical classes, positivity}
\approvals{Behrouz & yes \\ Benoît & yes \\ Stefan & yes}

Over a $ℚ$-factorial projective variety the determinant of any coherent sheaf
naturally defines an element of $N^1(X)_{ℚ}$.  To avoid potentially cumbersome
notations, let us fix a notation for such numerical classes.

\begin{notation}[Numerical classes]
  Let $X$ be a $ℚ$-factorial, projective variety and $\sF$ a coherent sheaf of
  $\sO_X$-modules.  Consider the Weil divisorial sheaf
  $\det \sF := (Λ^{\rank \sF} \sF)^{**}$ ---when $\sF$ is torsion and its rank
  is zero, then $\det \sF$ is nothing but the zero sheaf.  The numerical class
  of $\det \sF$ will be written as $[\sF] ∈ N^1(X)_{ℚ}$.
\end{notation}

\begin{warning}[Lack of additivity]
  Note that the numerical class operator $[\bullet]$ is not necessarily additive
  in exact sequences.  In fact, since the reflexive hull of any torsion sheaf is
  zero, the ideal sheaf sequence of any Cartier divisor will give a
  counterexample.
\end{warning}

\begin{notation}[Harder-Narasimhan filtration, generic positivity]
  Let $X$ be a normal, projective variety and $H$ be an ample Cartier divisor on
  $X$.  If $\sF$ is any torsion free, coherent sheaf of $\sO_X$-modules,
  consider the associated Harder-Narasimhan filtration
  $$
  0 = \sF_0 \subsetneq \sF_1 \subsetneq \cdots \subsetneq \sF_r = \sF.
  $$
  With this notation in place, write
  $$
  μ^{\max}_H(\sF) := μ_H(\sF_1) \quad\text{and}\quad μ^{\min}_H(\sF) :=
  μ_H(\sF_r/\sF_{r-1}).
  $$
  We call $\sF$ \emph{generically semipositive with respect to $H$} if
  $μ^{\min}_H(\sF) ≥ 0$.  We call $\sF$ \emph{generically semipositive} if $\sF$
  is generically semipositive with respect to any ample divisor.
\end{notation}

\part{Fractional semipositivity and application to hyperbolicity}
%
% Do not edit the following line.  The text is automatically updated by
% subversion.
%
\svnid{$Id: 03-fractDifferentials.tex 176 2016-10-25 12:24:02Z kebekus $}

\section{Logarithmic differentials with fractional pole order}
\label{sect:log-diffs}
\subversionInfo
\approvals{Behrouz & yes \\ Benoît & yes \\ Stefan & yes}

In this section we define the sheaves of \emph{adapted differential forms}.
These sheaves are, in a sense, the natural generalisation of sheaves of
log-differential forms for pairs $(X,D)$ with reduced boundary divisor $D$, to
the context of pairs $(X,Δ)$ with $Δ=\sum δ_i·Δ_i$, where $δ_i ∈ [0,1]∩ ℚ$ are
fractional.  Their construction depends on the choice of the adapted morphism.
Campana realised that, even in the purely logarithmic setting of Viehweg's
hyperbolicity problem, they provide great flexibility in dealing with birational
problems.  We begin this section by explaining the local description of these
sheaves when $(X,Δ)$ is snc.

\subsection{Informal explanation and local computation}
\label{sec:localComp}
\approvals{Behrouz & yes \\ Benoît & yes \\ Stefan & yes}

Throughout the present Section~\ref{sec:localComp}, we consider the following
particularly simple setting.

\begin{setting}[Setup and notation for Section~\ref{sec:localComp}]\label{setting:3-1}
  Let $(X,Δ)$ be an an snc pair.  Let $γ : Y → X$ be a cover that is adapted to
  $(X,Δ)$ and can locally be written in normal form.  Assume that
  $\supp (Δ+\Branch γ)$ and $\supp (γ^* Δ+\Ramification γ)$ are both smooth.
  Choose a point $\vec y ∈ Y$ and set $\vec x := γ(y)$.  Observe that if
  $\vec x ∈ \supp Δ$, then there exists exactly one component of $D ⊆ Δ$ that
  contains $\vec x$.  Let $δ$ be the coefficient of this component.  If
  $\vec x \not ∈ \supp Δ$, set $δ := 0$.

  We choose coordinates
  $$
  x_0, …, x_n ∈ \sO_{X,\vec x} \quad\text{and}\quad y_0, …, y_n ∈ \sO_{Y, \vec y},
  $$
  centred about $\vec x$ and $\vec y$, respectively, that present $γ$ in local
  normal form.  In particular, there exists a number $m$ such that
  $x_0 ◦ γ = y_0^m$.  If $δ = 1$ and if $γ$ happens to be unramified at
  $\vec y$, we may assume that locally near $\vec x$, the divisor $D$ is given
  as $\{ x_0 = 0\}$.
\end{setting}

We are interested in writing formal fractional-exponent-differential forms
$σ := x_0^{-δ} · dx_0$.  While this is not well-defined on $X$, one \emph{can}
write down the formal pull-back of $σ$ to $Y$; this will lead to the definition
of adapted differentials.  For the convenience of the reader, we discuss the
cases where $δ = 0$, where $0 < δ < 1$ and where $δ = 1$ separately.

\subsubsection{The case where $0 < δ < 1$}
\approvals{Behrouz & yes \\ Benoît & yes \\ Stefan & yes}

In this case, the divisor $D$ is necessarily contained in the branch locus of
$γ$, and locally given as $\{x_0 = 0\}$.  One may formally write
$$
(dγ) (σ) = y_0^{-m · δ} · d (y_0^m) = m · y_0^{(m-1)-mδ} · d
y_0.
$$
The assumption that $γ$ is adapted ensures that $m· δ$ is integral, and hence so
is the exponent of $y_0$.  The fact that $δ < 1$ ensures that the exponent is
not negative.  We aim to define a \emph{sheaf of adapted differentials}, in
symbols $Ω^1_{(X,Δ,γ)}$, as a subsheaf of $Ω^1_Y$, whose stalk at $\vec y$ is
generated by the forms
$$
y_0^{(m-1)-mδ} · d y_0 \quad\text{and}\quad dy_1,\, …,\, dy_n ∈ Ω^1_{Y, \vec y}.
$$

\begin{warning}
  It might seem tempting to take this as a definition for the sheaf of adapted
  differentials.  However, the following example shows that this is quite
  delicate.  Let $Z$ be a smooth variety and $H$ a smooth hypersurface on $Z$.
  Let $\vec z ∈ H$ be any and $z_0, …, z_n ∈ \sO_{Z, \vec z}$ a regular system
  of parameters, where $z_0$ generates the ideal of $H$.  Then note that the
  span
  $$
  \left\langle z_0^2 · d z_0,\, d z_1,\, …,\, d z_n \right\rangle ⊂ Ω^1_{Z,
    \vec z}
  $$
  \emph{does depend in a non-trivial way on the choice of coordinates}.  To give
  a proper definition, it will always be necessary to take the morphism $γ$ into
  account.
\end{warning}

In order to define $Ω^1_{(X,Δ,γ)}$ properly, in a coordinate-free way, we
compare its set of generators-to-be to the well-known set of generators for the
image of the pull-back map $dγ: γ^* Ω^1_X → Ω^1_Y$,
\begin{align*}
  Ω^1_{(X,Δ,γ),\vec y} & = \left\langle y_0^{(m-1)-mδ} · d y_0,\, dy_1,\, …,\, dy_n \right\rangle & ⊆ Ω^1_{Y, \vec y}\phantom{.}\\
  \Image(dγ)_{\vec y} & = \left\langle y_0^{m-1} · d y_0,\, dy_1,\, …,\, dy_n \right\rangle & ⊆ Ω^1_{Y, \vec y}.
\end{align*}
This suggests to define $Ω^1_{(X,Δ,γ)}$ near the point $\vec y$ in one of the
two following, equivalent ways,
\begin{align}
\label{eq:A} Ω^1_{(X,Δ,γ)} & := Ω^1_Y ⊗ \left( y_0^{(m-1)-mδ} \right) + \Image(dγ) \\
\label{eq:B} Ω^1_{(X,Δ,γ)} & := \left( \Image(dγ) ⊗ \bigl(y_0^{-mδ} \bigr)
\right) ∩ Ω^1_Y.
\end{align}

\begin{explanation}[Ideals in \eqref{eq:A}]
  In \eqref{eq:A}, we view $\bigl(y_0^{(m-1)-mδ}\bigr)$ as an ideal, view
  $Ω^1_Y ⊗ \bigl(y_0^{(m-1)-mδ} \bigr)$ as a subsheaf of $Ω^1_Y$, and the sum is
  the sum of coherent subsheaves there.
\end{explanation}

\begin{explanation}[Weil divisorial sheaves in \eqref{eq:B}]
  In \eqref{eq:B}, we view $\bigl(y_0^{-mδ} \bigr)$ as the Weil divisorial sheaf
  generated by the rational function $y_0^{-mδ}$, view $Ω^1_Y$ as a subsheaf of
  $\Image(dγ) ⊗ \bigl(y_0^{-mδ} \bigr)$, and the intersection is the
  intersection of coherent subsheaves there.
\end{explanation}

In order to avoid the awkward use of adapted coordinates, observe that the
divisor given by $y_0^{m-1}$ equals the ramification divisor of $γ$, while the
divisor given by $y_0^{mδ}$ is the pull-back divisor $γ^*Δ$.
Definitions~\eqref{eq:A}--\eqref{eq:B} thus simplify as follows,
\begin{align}
\label{eq:C} Ω^1_{(X,Δ,γ)} & := Ω^1_Y ⊗ \sO_Y \bigl( -γ^*Δ - \Ramification(γ)\bigr) + \Image(dγ) \\
\label{eq:D} Ω^1_{(X,Δ,γ)} & := \bigl( \Image(dγ) ⊗ \sO_Y( γ^*Δ) \bigr) ∩ Ω^1_Y.
\end{align}

\subsubsection{The case where $δ = 0$}
\approvals{Behrouz & yes \\ Benoît & yes \\ Stefan & yes}

In this case, $γ$ may or may not be ramified at $\vec y$.  The form
$σ = x_0^{-δ}· dx_0 = d x_0$ is an ordinary differential form, and so is its
pull-back $(dγ)(σ) = m · y_0^{m-1} · d y_0$.  We would then set
$$
Ω^1_{(X,Δ,γ)} := \Image(dγ)
$$ near the point $\vec y$.  Observe that this
definition agrees with \eqref{eq:C}--\eqref{eq:D} above.

\subsubsection{The case where $δ = 1$}
\approvals{Behrouz & yes \\ Benoît & yes \\ Stefan & yes}

In this case, $γ$ may or may not be ramified at $\vec y$.  If $γ$ \emph{is}
ramified at $\vec y$, then the assumption that $\supp (Δ+\Branch γ)$ is smooth
implies that near $\vec x$, the divisor $D$ equals the branch locus, and is
given as $\{x_0=0\}$.  The form $σ = x_0^{-1}· dx_0 = d \log x_0$ is a
logarithmic differential form, and so is its pull-back $(dγ)(σ) = d \log y_0$.
In this case, we would like to define the sheaf of adapted differentials near
$\vec y$ as
$$
Ω^1_{(X,Δ,γ)} := Ω^1_Y(\log Δ_γ) = γ^* Ω^1_X(\log \lfloor Δ \rfloor),
\quad\text{where } Δ_γ := (γ^* \lfloor Δ \rfloor )_{\red}.
$$
Formulas \eqref{eq:C}--\eqref{eq:D} include this case after the following minor
adjustment.  In fact, extending the pull-back morphism $dγ$ to include
logarithmic differentials, $dγ : γ^* Ω^1_X(\log \lfloor Δ \rfloor) → Ω^1_Y(\log
Δ_γ)$, we can write
\begin{align*}
  Ω^1_{(X,Δ,γ)} & := Ω^1_Y(\log Δ_γ) ⊗ \sO_Y \Bigl( -γ^*\{Δ\} - \Ramification(γ)\Bigr) + \Image(dγ) \\
  Ω^1_{(X,Δ,γ)} & := \Bigl( \Image(dγ) ⊗ \sO_Y \bigl( γ^*\{Δ\} \bigr) \Bigr) ∩ Ω^1_Y(\log Δ_γ).
\end{align*}
These formulas will re-appear in the succeeding Section~\ref{ssec:defns}, where adapted
differentials are formally introduced.

\subsection{Formal definition}
\label{ssec:defns}
\approvals{Behrouz & yes \\ Benoît & yes \\ Stefan & yes}

We now give a formal and coordinate-free definition of ``adapted
differentials'', following the discussion of the previous subsection.  A local
description is also included.

\subsubsection{Adapted differentials for a good cover}
\label{ssec:adgagc}
\approvals{Behrouz & yes \\ Benoît & yes \\ Stefan & yes}

We define adapted differentials first for covers that satisfy all the
assumptions of Setting~\ref{setting:3-1}.  We call such covers \emph{good}.  To
be more precise, the following definition will be used.

\begin{defn}\label{defn:goodcover}
  Let $(X,Δ)$ be a pair, and $γ: Y → X$ be a cover.  The cover is called
  \emph{good} if the following properties hold.
  \begin{enumerate}
  \item The variety $X$ and its subvariety $\supp (Δ+\Branch γ)$ are smooth.
  \item The variety $Y$ and its subvariety $\supp (γ^* Δ+\Ramification γ)$ are
    smooth.
  \item The cover $γ$ is adapted.
  \item The cover $γ$ can locally be written in normal form.
  \end{enumerate}
\end{defn}

As indicated above, we take \eqref{eq:D} as the definition of adapted
differentials, at least for good covers.

\begin{defn}[Adapted differentials for good cover]\label{def:E:6a}
  Let $(X,Δ)$ be a pair, and $γ: Y → X$ a be good cover.  Consider the pull-back
  map of logarithmic differentials,
  $$
  dγ : γ^* Ω^1_X \bigl(\log \lfloor Δ \rfloor \bigr) → Ω^1_Y(\log Δ_γ), \quad
  \text{where }Δ_γ := (γ^* \lfloor Δ \rfloor )_{\red}.
  $$
  The \emph{sheaf of adapted differentials on $Y$} is then defined as
  $$
  Ω^1_{(X,Δ,γ)} := \Bigl( \Image (dγ) ⊗ \sO_Y\bigl( γ^*\{Δ\} \bigr) \Bigr) \:∩\: Ω^1_Y(\log Δ_γ),
  $$
  where the intersection is the intersection of subsheaves in
  $Ω^1_Y (\log Δ_γ) ⊗ \sO_Y(γ^* \{ Δ\})$.
\end{defn}

\begin{rem}[Inclusions between sheaves of adapted differentials]\label{rem:X0}
  In the setting of Definition~\ref{def:E:6a}, it follows immediately from the
  definition that there exist inclusions,
  \begin{equation}\label{eq:X0}
    γ^* Ω^1_X(\log \lfloor Δ \rfloor) \;⊆\; Ω^1_{(X,Δ,γ)} \;⊆\; Ω^1_Y \bigl(
    \log Δ_γ \bigr),
  \end{equation}
  satisfying the following properties.
  \begin{enumerate}
  \item\label{il:X0-1} The first inclusion in \eqref{eq:X0} is an equality away
    from $\supp γ^* \{Δ\}$.
  \item\label{il:X0-2} The three terms of \eqref{eq:X0} are equal away from
    $\supp \Ramification(f)$, and near $\supp Δ_γ$.
  \item\label{il:X0-3} If the covering morphism $γ$ is Galois, say with group
    $G$, then all sheaves appearing in \eqref{eq:X0} carry natural structures of
    $G$-sheaves, and the inclusions are inclusions of $G$-sheaves.
  \end{enumerate}
\end{rem}

\begin{rem}[Local description of adapted differentials]
  The inclusions \eqref{eq:X0} can be written down in local coordinates, near
  any given point $\vec y ∈ Y$.  If $γ$ is étale at $\vec y$, or if
  $\vec y ∈ \supp Δ_γ$ then all three sheaves agree, and there is nothing much
  to do.  Let us therefore assume that
  $\vec y ∈ \Ramification(f) \setminus \supp Δ_γ$.  Choose local coordinates as
  in Setup~\ref{setting:3-1} and follow the notation introduced there.

  Near $\vec y$ the sheaves $Ω^1_Y(\log Δ_γ)$ and $Ω^1_Y$ agree, and so do
  $γ^* Ω^1_X(\log \lfloor Δ \rfloor)$ and $γ^* Ω^1_X$.  The sheaf $Ω^1_Y$ is
  freely generated as an $\sO_Y$-module by symbols $d y_0, …, d y_n$.  The
  sheaves of \eqref{eq:X0} are then generated as follows,
  $$
  \begin{matrix}
    γ^* Ω^1_X & = & \big\langle & y_0^{m-1}·dy_0, & dy_1, & …, & dy_n & \big\rangle \\
    Ω^1_{(X,Δ,γ)} & = & \big\langle & y_0^{(m-1)-mδ}·dy_0, & dy_1, & …, & dy_n & \big\rangle \\
    Ω^1_Y & = & \big\langle & d y_0, & d y_1, & …, & d y_n & \big\rangle.
  \end{matrix}
  $$
\end{rem}

The following is now an immediate consequence of the local description.

\begin{cor}[Determinants and Chern classes of adapted differentials]\label{cor:det}
  In the setting of Definition~\ref{def:E:6a}, we have equalities of sheaves,
  \begin{align*}
    \det\: Ω^1_{(X,Δ,γ)} & = \det \left(γ^* \; Ω^1_X(\log \lfloor Δ \rfloor) \right) ⊗ \det \sO_Y\bigl( γ^* \{ Δ\} \bigr) \\
                           & = \sO_Y \bigl( γ^*(K_X+ \lfloor Δ \rfloor ) + γ^* \{ Δ\} \bigr) \\
                           & = \sO_Y \bigl( γ^*(K_X+ Δ) \bigr).
  \end{align*}
  In particular, $c_1 \bigl( Ω^1_{(X,Δ,γ)} \bigr) = [γ^*(K_X+ Δ)]$.  \qed
\end{cor}

\subsubsection{Adapted reflexive differentials in the general setting}
\approvals{Behrouz & yes \\ Benoît & yes \\ Stefan & yes}

We now extend the definition of the adapted differentials from good covers to
arbitrary ones.

\begin{defn}[Adapted reflexive differentials]\label{def:E:6b}
  Let $(X,Δ)$ be a pair, and let $γ: Y → X$ be a cover that is adapted to
  $(X,Δ)$.  Let $X° ⊆ X$ and $Y° ⊆ Y$ be the maximal open sets where $f$ can
  locally be written in normal form and where $\supp (Δ+\Branch γ)$ and
  $\supp (γ^* Δ+\Ramification γ)$ are both smooth; these are big open subsets of
  $X$ and $Y$, respectively.  Let $ι : Y° → Y$ be the inclusion map, and set
  $Δ° := Δ|_{X°}$ and $γ° := γ|_{Y°}$.  We define the \emph{sheaf of adapted
    reflexive differentials on $Y$} as
  $$
  Ω^{[1]}_{(X,Δ,γ)} := ι_* \: Ω^1_{(X°,Δ°,γ°)}.
  $$
  where $Ω^1_{(X°,Δ°,γ°)}$ is the sheaf that has been introduced in
  Definition~\vref{def:E:6a}.
\end{defn}

\begin{rem}[Inclusions between sheaves of adapted reflexive differentials]\label{rem:X1}
  In the setting of Definition~\ref{def:E:6b}, observe that $Ω^1_{(X°,Δ°,γ°)}$
  is locally free on $Y°$.  It follows that the sheaf of adapted reflexive
  differentials is reflexive.  Using \eqref{eq:X0}, push-forward from $Y°$ to
  $Y$ induces inclusions of reflexive sheaves as follows,
  \begin{equation}\label{eq:X1}
    γ^{[*]} Ω^{[1]}_X(\log \lfloor Δ \rfloor) \;⊆\; Ω^{[1]}_{(X,Δ,γ)} \;⊆\; Ω^{[1]}_Y \bigl(
    \log Δ_γ \bigr).
  \end{equation}
  where $Δ_γ := (γ^* \lfloor Δ \rfloor )_{\red}$.  Remark~\ref{rem:X0} and
  Corollary~\ref{cor:det} also have direct analogues.
  \begin{enumerate}
  \item\label{il:X1-1} The first inclusion in \eqref{eq:X1} is an equality away
    from $\supp γ^* \{Δ\}$.
  \item\label{il:X1-2} The three terms of \eqref{eq:X1} are equal away from
    $\supp \Ramification(f)$, and near general points of $\supp Δ_γ$.
  \item\label{il:X1-3} If the covering morphism $γ$ is Galois, say with group
    $G$, then all sheaves appearing in \eqref{eq:X1} carry natural structures of
    $G$-sheaves, and the inclusions are inclusions of $G$-sheaves.
  \item\label{il:X1-4} We have an equality of sheaves,
    $$
    \det\: Ω^{[1]}_{(X,Δ,γ)} = \sO_Y \bigl( γ^*(K_X+ Δ) \bigr).
    $$
    If $K_X+Δ$ is $ℚ$-factorial, then
    $c_1 \bigl( Ω^1_{(X,Δ,γ)} \bigr) = [γ^*(K_X+ Δ)]$.
  \end{enumerate}
\end{rem}

\subsubsection{Relation to earlier definitions}
\approvals{Behrouz & yes \\ Benoît & yes \\ Stefan & yes}

If $(X,Δ)$ is a $\mathcal C$-pair, the notion of ``adapted differentials''
agrees with the notion introduced in earlier papers of Campana \emph{et al.},
cf.~\cite[Sect.~2.D]{MR2860268} and Campana's work referenced there.

%
% Do not edit the following line.  The text is automatically updated by
% subversion.
%
\svnid{$Id: 04-tangentsAndFoliations.tex 169 2016-10-18 13:42:18Z kebekus $}

\section{Fractional tangents and foliations}
\label{sect:fractional}
\subversionInfo
\approvals{Behrouz & yes \\ Benoît & yes \\ Stefan & yes}

The aim of this section is to lay down the technical groundwork for
Section~\ref{sect:failure}.  There, we construct a certain subsheaf of $\sT_X$
and study its integrability properties.  The key technical result here is
Proposition~\ref{prop:sddffg}, whose proof requires some preliminary
observations about the local description of vector fields that are transversal
(resp.\ tangential) to the branch locus of an adapted cover.

\subsection{Adapted tangents}
\label{subsect:local}
\approvals{Behrouz & yes \\ Benoît & yes \\ Stefan & yes}

We first define, in the obvious way, the notion of an adapted tangent sheaf, by
dualizing the adapted sheaf of differentials.

\begin{defn}[Adapted tangents]\label{def:E:7}
  Given a pair $(X,Δ)$ and an adapted cover $γ: Y → X$, set
  $$
  \sT_{(X,Δ,γ)} := \bigl( Ω^{[1]}_{(X,Δ,γ)} \bigr)^*,
  $$
  where $Ω^{[1]}_{(X,Δ,γ)}$ is the sheaf of adapted reflexive differentials that
  was introduced in Definition~\ref{def:E:6b}.  We call $\sT_{(X,Δ,γ)}$ the
  \emph{adapted tangent sheaf} or \emph{sheaf of adapted tangents}.
\end{defn}

\begin{rem}[Inclusions between tangent sheaves]\label{rem:X2}
  Dualising, Remark~\ref{rem:X1} yields inclusions
  \begin{equation}\label{eq:X2}
    \sT_Y \bigl( -\log Δ_γ \bigr) \;⊆\; \sT_{(X,Δ,γ)} \;⊆\;
    γ^{[*]} \sT_X(-\log \lfloor Δ \rfloor).
  \end{equation}
  where $Δ_γ := (γ^* \lfloor Δ \rfloor )_{\red}$.  The following additional
  properties hold.
  \begin{enumerate}
  \item\label{il:X2-1} The second inclusion in \eqref{eq:X2} is an equality away
    from $\supp γ^* \{Δ\}$.
  \item\label{il:X2-2} The three terms of \eqref{eq:X2} are equal away from
    $\supp \Ramification(f)$, and near general points of $\supp Δ_γ$.
  \item\label{il:X2-3} If the covering morphism $γ$ is Galois, say with group
    $G$, then all sheaves appearing in \eqref{eq:X2} carry natural structures of
    $G$-sheaves, and the inclusions are inclusions of $G$-sheaves.
  \item\label{il:X2-4} We have an equality of sheaves,
    $$
    \det\: \sT_{(X,Δ,γ)} = \sO_Y \bigl( -γ^*(K_X+ Δ) \bigr).
    $$
    If $K_X+Δ$ is $ℚ$-factorial, then
    $c_1 \bigl( \sT_{(X,Δ,γ)} \bigr) = [-γ^*(K_X+ Δ)]$.
  \end{enumerate}
\end{rem}

\subsubsection{Adapted tangents for good covers}
\approvals{Behrouz & yes \\ Benoît & yes \\ Stefan & ---}

For a good adapted cover, we were able to give a complete descriptions of
adapted differentials in Section~\ref{ssec:adgagc}.  The following is the direct
analogue of this for adapted tangents.

\begin{rem}[Local description of adapted tangents]\label{rem:localtgts}
  Setup as in Definition~\ref{def:E:7}.  If the cover $γ$ is good, then the
  inclusions \eqref{eq:X2} can be written down in local coordinates, near any
  given point $\vec y ∈ Y$.  If $γ$ is étale at $\vec y$, or if
  $\vec y ∈ \supp Δ_γ$ then all three sheaves agree, and there is nothing much
  to do.  Let us therefore assume that
  $\vec y ∈ \Ramification(f) \setminus \supp Δ_γ$.  Choose local coordinates as
  in Setup~\ref{setting:3-1} and follow the notation introduced there.

  Near $\vec y$, the sheaves $γ^* \sT_X(- \log \lfloor Δ \rfloor)$ and
  $γ^* \sT_X$ agree, and so do the sheaves
  $\sT_Y \bigl( -\log (γ^* \lfloor Δ \rfloor)_{\red} \bigr)$ and $\sT_Y$.  The
  sheaf $γ^* \sT_X$ is freely generated as an $\sO_Y$-module by symbols
  $γ^* \frac{\partial}{\partial x_0}, …, γ^* \frac{\partial}{\partial x_n}$.
  The sheaves of \eqref{eq:X2} are then generated as follows,
  \begin{equation}\label{eq:cvbcu}
    \begin{matrix}
      γ^* \sT_X & = & \big\langle & γ^* \frac{\partial}{\partial x_0}, & γ^* \frac{\partial}{\partial x_1}, & …, & γ^* \frac{\partial}{\partial x_n} & \big\rangle \\
      \sT_{(X,Δ,γ)} & = & \big\langle & y_0^{mδ}·γ^* \frac{\partial}{\partial x_0}, & γ^* \frac{\partial}{\partial x_1}, & …, & γ^* \frac{\partial}{\partial x_n} & \big\rangle \\
      \sT_Y & = & \big\langle & y_0^{m-1}·γ^* \frac{\partial}{\partial x_0}, & γ^* \frac{\partial}{\partial x_1}, & …, & γ^* \frac{\partial}{\partial x_n} & \big\rangle.
    \end{matrix}
  \end{equation}
\end{rem}

The local description has the following consequence, which will be relevant in
the study of foliations.

\begin{lem}\label{lem:trafalgar}
  In the setting of Definition~\ref{def:E:7}, let
  $\vec V ∈ H^0 \bigl( X,\, \sT_X(-\log \lfloor Δ \rfloor) \bigr)$ be any
  logarithmic vector field on $X$, with associated pull-back
  $$
  γ^* \vec V ∈ H^0 \bigl( Y,\, γ^* \sT_X(-\log \lfloor Δ \rfloor) \bigr).
  $$
  Then, the following holds.
  \begin{enumerate}
  \item\label{il:nelson} If $\vec V$ is everywhere transversal to the smooth
    subvariety $\supp (\{Δ\} + \Branch γ)$, then $γ^* \vec V$ generates the
    quotient $γ^* \sT_X(-\log \lfloor Δ \rfloor) / \sT_{(X,Δ,γ)}$.

  \item\label{il:napoleon} If $\vec V$ is everywhere tangential to the smooth
    subvariety $\supp (\{Δ\} + \Branch γ)$, then $γ^* \vec V$ is contained in
    $H^0 \bigl( Y,\, \sT_Y ( -\log Δ_γ ) \bigr)$.
  \end{enumerate}
\end{lem}
\begin{proof}
  Both items can be shown locally, near given points $\vec y ∈ Y$.  Again, we
  choose local coordinates as in Setup~\ref{setting:3-1} and follow the notation
  introduced there.  Write the vector field $\vec V$ and its pull-back locally
  as
  $$
  \vec V = \sum_{i=0}^n f_i·\frac{\partial}{\partial x_i} \quad\text{and}\quad
  γ^* \vec V = \sum_{i=0}^n (f_i◦ γ)·γ^* \frac{\partial}{\partial x_i}
  $$
  where $f_i ∈ \sO_{X,\vec x}$.

  \subsubsection*{Proof of \ref{il:nelson}}
  
  If $\vec y$ is not contained in $\supp γ^* \{Δ\}$, then we have seen in
  Item~\ref{il:X2-1} of Remark~\ref{rem:X2} that
  $\sT_{(X,Δ,γ)} = γ^* \sT_X(-\log \lfloor Δ \rfloor)$, so there is nothing to
  show.  Let us therefore assume that $\vec y ∈ \supp γ^* \{Δ\}$.  This allows
  to use the local description of adapted tangent from
  Remark~\ref{rem:localtgts}.  The assumption that $\vec V$ is transversal
  implies that $f_0$ does not vanish at $\vec x$.  But then $f_0 ◦ γ$ will not
  vanish at $\vec y$, and the explicit description in \eqref{eq:cvbcu} yields
  the claim.

  \subsubsection*{Proof of \ref{il:napoleon}}

  If $\vec y$ is contained in $\supp Δ_γ$ or in the complement of
  $\supp \Ramification(f)$, then we have seen in Item~\ref{il:X2-2} of
  Remark~\ref{rem:X2} that the three sheaves of \eqref{eq:X2} agree, and there
  is nothing to show.  We will therefore assume without loss of generality that
  $\lfloor Δ \rfloor = 0$, and that $\vec y ∈ \Ramification(f)$.  The assumption
  that $\vec V$ is tangential implies that $f_0$ vanishes along $\{x_0 = 0\}$,
  so that $f_0 = x_0·g_0$ and
  $$
  γ^* \vec V = y_0^m·(g_0◦ γ)·γ^* \frac{\partial}{\partial x_0} +
  \sum_{i=1}^n (f_i◦ γ)·γ^* \frac{\partial}{\partial x_i}.
  $$
  Again, a comparison with the explicit description in \eqref{eq:cvbcu} yields
  the claim.
\end{proof}

\subsection{Foliations}
\approvals{Behrouz & yes \\ Benoît & yes \\ Stefan & yes}

In this subsection we use the local machinery developed in
Subsection~\ref{subsect:local} to establish a technical tool that will play a
significant role in the proof of Theorem~\ref{thm:orbSemPos}.

\subsubsection{Lifting the O'Neil tensor}
\label{ssec:Neil}
\approvals{Behrouz & yes \\ Benoît & yes \\ Stefan & yes}

A key problem in the proof of the main semipositivity result for a given pair
$(X,Δ)$ is to relate the integrability of a certain subsheaf $\sF$ of $\sT_X$ to
the behaviour of the pull-back of the O'Neil tensor\footnote{As introduced in Notation \ref{not:ONeil}.} on $γ^{[*]}\sF$ and
$\sG_{γ}:=(γ^{[*]}\sF∩ \sT_{(X,Δ,γ)})$.  In the next proposition we show that
the restriction of the lift of the O'Neil tensor maps the smaller sheaf
$\sG_{γ}$ to $\sT_{(X,Δ,γ)}/\sG_{γ}$, after taking reflexive hulls.

\begin{prop}[Lifting the O'Neil tensor]\label{prop:sddffg}
  Let $(X,Δ)$ be a pair and let $γ: Y → X$ be an adapted cover.  Let
  $\sF ⊆ \sT_X(- \log \lfloor Δ \rfloor)$ be a saturated subsheaf.  Consider
  reflexive sheaves on $\sG_Y$, $\sG_{γ}$ and $\sG$ on $Y$ as follows,
  \begin{equation}\label{eq:cmnfgj}
    \underbrace{\sT_Y \bigl( -\log Δ_γ \bigr)}_{\text{contains }\sG_Y \::=\: \sG ∩
      \sT_Y} \quad ⊆ \quad \underbrace{\sT_{(X,Δ,γ)}}_{\mathclap{\text{contains
        }\sG_γ \::=\: \sG ∩ \sT_Y}} \quad ⊆ \quad \underbrace{γ^{[*]} \sT_X(-\log
      \lfloor Δ \rfloor)}_{\text{contains }\sG \::=\: γ^{[*]} \sF}.
  \end{equation}
  Next, consider the O'Neil tensor
  $$
  N : \sF^{[2]} → \Bigl(\factor{\sT_X(- \log \lfloor Δ \rfloor)}{\sF}
  \Bigr)^{**},
  $$
  its reflexive pull-back,
  $$
  γ^{[*]}N : \sG^{[2]} → γ^{[*]} \Bigl(\factor{\sT_X(- \log \lfloor Δ \rfloor)}{\sF}
  \Bigr),
  $$
  and write $N_γ$ for the restriction of $γ^{[*]}N$ to the subsheaf
  $\sG_γ^{[2]} ⊆ \sG^{[2]}$.  Then $N_γ$ factorises as follows,
  $$
  \xymatrix{ %
    \sG_γ^{[2]} \ar[r]_(.3){α} \ar@/^7mm/[rr]^{N_γ} &
    \bigl(\factor{\sT_{(X,Δ,γ)}}{\sG_γ} \bigr)^{**} \ar@{^(->}[r]_(.4){β} &
    γ^{[*]} \Bigl( \factor{\sT_X(- \log \lfloor Δ \rfloor)}{\sF} \Bigr).  }
  $$
\end{prop}
\begin{proof}
  This is easier than the involved notation suggests.  An elementary diagram
  chase shows that the natural morphism $β$ is injective.

  \subsubsection*{Step 1: Simplification}
  \approvals{Behrouz & yes \\ Benoît & yes \\ Stefan & yes}

  To prove that a morphism of reflexive sheaves factorises via a third, it
  suffices to prove the existence of a factorisation on a big open set.  We may
  therefore assume that the following additional properties hold.
  \begin{enumerate}
  \item The cover $γ$ is good.
  \item The saturated subsheaf $\sF ⊆ \sT_X(- \log \lfloor Δ \rfloor)$ is
    actually a subbundle.
  \item The subsheaf $\sG_Y ⊆ \sT_Y \bigl( -\log Δ_γ \bigr)$ is a
    subbundle.  Ditto for $\sG_Y ⊆ \sT_{(X,Δ,γ)}$ and
    $\sG ⊆ γ^{[*]} \sT_X(-\log \lfloor Δ \rfloor)$.
  \item If $D ⊆ \supp (Δ+\Branch γ)$ is any component, then either $\sF$ is
    everywhere transversal to $D$, or everywhere tangent to $D$.
  \end{enumerate}
  This simplifies notation greatly, as all sheaves in question will be locally
  free, so there is no need to take reflexive hulls in each step.  There is more
  we can do.  Recalling from Item~\ref{il:X2-2} of Remark~\ref{rem:X2} that the
  three terms of \eqref{eq:cmnfgj} are equal near general points of $\supp Δ_γ$,
  so that the claim of the proposition is certainly true there, we can also
  assume that the following holds.
  \begin{enumerate}
  \item The integral part of $Δ$ is empty, so $\lfloor Δ \rfloor = 0$ and
    $Δ_γ = 0$.
  \end{enumerate}
  Again, this simplifies notation substantially, allowing us to drop all
  ``$\log \lfloor Δ \rfloor$'' and ``$\log Δ_γ$'' from sheaves of tangents and
  differentials.

  \subsubsection*{Step 2: reduction to the local case}
  \approvals{Behrouz & yes \\ Benoît & yes \\ Stefan & yes}

  The statement of the proposition is clearly local; it suffices to prove the
  factorisation in an analytic neighbourhood of any given point $\vec y ∈ Y$.
  Again, if $γ$ is étale at $\vec y$, then Item~\ref{il:X2-2} of
  Remark~\ref{rem:X2} that the three terms of \eqref{eq:cmnfgj} are equal, and
  there is nothing to show.  We will therefore assume that $γ$ is ramified at
  $\vec y$.  Set $\vec x = γ(\vec y)$ and let $D ⊆ \supp (Δ+\Branch γ)$ be the
  unique component that contains $\vec x$ ---the component is unique because
  $\supp (Δ+\Branch γ)$ is smooth by the assumption that $γ$ is a good cover.
  Replacing $X$ with a suitably small neighbourhood of $\vec x$, if need be, we
  can assume that the following holds in addition.
  \begin{enumerate}
  \item The sheaf $\sF$ is free, say generated by global sections $σ_1$, …,
    $σ_r$.
  \item The support of $Δ+\Branch γ$ is irreducible, $D = \supp (Δ+\Branch γ)$.
  \item If $\sF$ is everywhere transversal to $D$, then $σ_1$ is everywhere
    transversal to $D$.
  \end{enumerate}

  \subsubsection*{Step 3: Proof in case that $\sF$ is everywhere transversal to $D$}
  \approvals{Behrouz & yes \\ Benoît & yes \\ Stefan & yes}

  The following diagram summarises the sheaves in question.
  $$
  \xymatrix{ %
    0 \ar[r] & \sG_γ \ar[r] \ar@{^(->}[d] & \sG \ar[r] \ar@{^(->}[d] & \factor{\sG}{\sG_γ} \ar[r] \ar[d]^{η} & 0 \\
    0 \ar[r] & \sT_{(X,Δ,γ)} \ar[r] & γ^* \sT_X \ar[r] & \factor{γ^* \sT_X}{\sT_{(X,Δ,γ)}} \ar[r] & 0 \\
  }
  $$
  The morphism $η$ is injective by construction.  Lemma~\ref{lem:trafalgar}
  asserts that it is also surjective.  More is true: Item~\ref{il:nelson} even
  asserts that the image is generated by the class of the section $γ^* σ_1$.
  The snake lemma thus implies that the natural map
  $$
  \factor{\sT_{(X,Δ,γ)}}{\sG_γ} → \factor{γ^* \sT_X}{\sG}
  $$
  is isomorphic.  The question of factorisation is therefore void and
  Proposition~\ref{prop:sddffg} is shown in case that $\sF$ is everywhere
  transversal.

  \subsubsection*{Step 4: Proof in case that $\sF$ is everywhere tangential to $D$}
  \approvals{Behrouz & yes \\ Benoît & yes \\ Stefan & yes}

  If $\sF$ is everywhere tangential to $D$, then Lemma~\ref{lem:trafalgar}
  asserts that $\sG$ is already contained in $\sT_Y$, so that the sheaves $\sG$,
  $\sG_γ$ and $\sG_Y$ are actually equal.  The composed morphism $μ$,
  $$
  \xymatrix{ %
    \sG_Y^{[2]} \ar[rrr]_{\text{O'Neil tensor on }Y} \ar@/^5mm/[rrrrrr]^{μ} &&&
    \factor{\sT_Y}{\sG_Y} \ar[rrr] &&& \factor{γ^* \sT_X}{\sG} %
  }
  $$
  clearly equals $N_γ$ over the open set where $γ$ is étale.  Since all sheaves
  in question are locally free, hence torsion free, this means that $N_γ$ equals
  the map $μ$ everywhere.  But $μ$ factors as desired.  This shows
  Proposition~\ref{prop:sddffg} in the last remaining case.
\end{proof}

\subsubsection{Chern classes}
\approvals{Behrouz & yes \\ Benoît & yes \\ Stefan & yes}

When proving the main result of this chapter, we will need to pull-back
foliations from the variety $X$ to an adapted cover, saturate, intersect and
keep track of how Chern classes change in the process.  The following somewhat
technical lemma summarises everything that we will need.

\begin{prop}[Chern classes of foliations]\label{prop:xad}
  Let $(X,Δ)$ be a pair and let $γ: Y → X$ be an adapted cover.  Further, let
  $\sF ⊆ \sT_X$ be a foliation.  Write $Δ = Δ^{trans} + Δ^{ntrans}$, as in
  Notation~\ref{not:genTrans} and consider the following sheaves.
  \begin{align*}
    \sG & := \sF ∩ \sT_X(- \log \lfloor Δ \rfloor) && \text{… saturated subsheaf of } \sT_X(-\log \lfloor Δ \rfloor) \\
    \sH & := γ^{[*]} \sG && \text{… saturated subsheaf of } γ^{[*]} \sT_X(-\log \lfloor Δ \rfloor) \\
    \sH_{(X,Δ,γ)} & := \sH ∩ \sT_{(X,Δ,γ)} && \text{… saturated subsheaf of } \sT_{(X,Δ,γ)}
  \end{align*}
  The determinant sheaves are then related as follows.
  \begin{align}
    \label{il:Y0} \det \sG & = \left[\det (\sF) ⊗ \sO_X( -\lfloor Δ^{trans} \rfloor ) \right]^{**}\\
    \label{il:Y1} \det \sH & = \left[\det \bigl( \sH_{(X,Δ,γ)} \bigr) ⊗ \sO_Y \bigl(γ^* \{ Δ^{trans} \} + \text{effective} \bigr) \right]^{**} \\
    \intertext{In summary,}
    \label{il:Y2} γ^{[*]} \det \sF & = \left[ \det \bigl(\sH_{(X,Δ,γ)} \bigr) ⊗ \sO_Y \bigl(γ^* Δ^{trans} + \text{effective} \bigr) \right]^{**}.
  \end{align}
\end{prop}
\begin{proof}
  Equation~\eqref{il:Y0} is elementary.  Equation~\eqref{il:Y2} is a combination
  of \eqref{il:Y0} and \eqref{il:Y1}.  It remains to prove \eqref{il:Y1}.  As an
  equation between reflexive sheaves, \eqref{il:Y1} can therefore be checked on
  a big open set.  We may therefore assume that $γ$ is a good cover.  Recalling
  from Item~\ref{il:X2-1} of Remark~\ref{rem:X2} that the sheaves
  $\sT_{(X,Δ,γ)}$ and $γ^{[*]} \sT_X$ agree away from the support of
  $γ^* \{Δ\}$, it will suffice to understand the difference between $\sH$ and
  its subsheaf $\sH_{(X,Δ,γ)}$ near a given point $\vec y$ in the support of
  $γ^* \{Δ^{trans}\}$.  There, the statement follows from the local description
  \eqref{eq:cvbcu} of Remark~\ref{rem:localtgts}, using Item~\ref{il:nelson} of
  Lemma~\ref{lem:trafalgar} as we have done in the proof of
  Proposition~\ref{prop:sddffg}.
\end{proof}

%
% Do not edit the following line.  The text is automatically updated by
% subversion.
%
\svnid{$Id: 05-mainResult.tex 189 2016-10-28 12:52:40Z taji $}

\section{Fractional semipositivity}
\label{sect:semipositivity}
\subversionInfo
\approvals{Behrouz & yes \\ Benoît & yes \\ Stefan & yes}

A celebrated result of Miyaoka, \cite[Cor.~8.6]{Miy85}, shows that for a smooth
projective variety $X$ (or more generally a normal projective variety with only
canonical singularities) positivity properties of the canonical sheaf $ω_X$ are
deeply related to those of the sheaf of (pluri-)differential forms.  More
precisely, if $K_X$ is pseudo-effective, then, for every positive integer $m$
and ample divisor $H⊂ X$, the sheaf $(Ω_X^1)^{⊗ m}$ is semipositive with
respect to $H$.  In other words, $c_1(\sQ)·[H]^{n-1} ≥ 0$, for all coherent
quotients $\sQ$ of $(Ω_X^1)^{⊗ m}$.  When $c_1(X)=0$ or $<0$, these results
can be traced back to Yau's theorem on the existence of Kähler-Einstein metrics,
\cite{MR0451180}.  Miyaoka's approach, on the other hand, is purely algebraic
and involves deep and delicate characteristic $p$ methods.  Campana and Păun
extend Miyaoka's results to the context of pairs.  Before we state their result
in Theorem~\ref{thm:orbSemPos} below, we recall the definition of generic
semipositivity with respect to an adapted cover.

\begin{defn}[\protect{Generic semipositivity w.r.t.\ an adapted cover, cf.~\cite[Def.~1.7]{CP13}}]
  Let $(X,Δ)$ be a projective pair where $X$, and let $γ : Y → X$ be an adapted
  morphism that is Galois with group $G$.  We say that $Ω^{[1]}_{(X,Δ,γ)}$ is
  \emph{$γ$-generically semipositive} if it is generically semipositive with
  respect to any ample divisor of the form $γ^*(\text{ample})$.
\end{defn}

\begin{rem}[Intersection with nef classes]\label{rem:polbynef}
  Nef divisors are limits of ample divisors.  If $Ω^{[1]}_{(X,Δ,γ)}$ is
  $γ$-generically semipositive, continuity of intersection products therefore
  implies that $c_1(\sQ)·[γ^*N]^{n-1} ≥ 0$, for all coherent quotients $\sQ$ of
  $(Ω^{[1]}_{(X,Δ,γ)})^{⊗ m}$ and all nef divisors $N ∈ ℚ\Div(X)$.
\end{rem}

\begin{thm}[\protect{Generic semipositivity of $Ω^{[1]}_{(X,Δ,γ)}$, cf.~\cite[Thm.~2.1]{CP13}}]\label{thm:orbSemPos}
  Let $(X,Δ)$ be a log canonical, projective pair, and let $γ : Y → X$ be an
  adapted cover that is Galois with group $G$.  If $K_X+Δ$ is pseudo-effective,
  then $Ω^{[1]}_{(X,Δ,γ)}$ is $γ$-generically semipositive.
\end{thm}

Part II of this chapter is devoted to a proof of Theorem~\ref{thm:orbSemPos}.
After some preparatory sections, the proof is given in
Section~\vref{ssec:potosp}.

%
% Do not edit the following line.  The text is automatically updated by
% subversion.
%
\svnid{$Id: 06-hyperbolicity.tex 189 2016-10-28 12:52:40Z taji $}

\section{Application to hyperbolicity}
\label{sect:hyperbol}
\subversionInfo
\approvals{Behrouz & yes \\ Benoît & yes \\ Stefan & yes}

As recalled in the introduction, the first step in proving Viehweg's
hyperbolicity Conjecture~\ref{conj:Vhypo} was carried out by Viehweg and Zuo in
Theorem~\ref{thm:VZ}, where it was shown that the variation in a smooth family
of canonically-polarised manifolds $f°: X° → Y°$ manifests itself as a
birationally-positive subsheaf $\sL$ of pluri-log forms.  The aim of this
section is to use Theorem~\ref{thm:orbSemPos} to extract bigness of the log
canonical divisor $K_Y+D$ from the positivity of $\sL$.  This is the content of
the following theorem of Campana and Păun.

\begin{thm}[\protect{Existence of pluri-log-canonical forms, cf.~\cite[Thm.~4.1]{CP13}}]\label{thm:CP-hypo}
  Let $(X,D)$ be a projective snc pair, where $D$ is reduced.  Let $N ∈ ℕ^+$ be
  a positive integer and let
  $$
  \sL ⊆ \bigotimes^N \, Ω^1_X\log (D)
  $$
  be an invertible subsheaf.  If $κ(X, \sL)= \dim X$, then $κ(K_X+D)=\dim X$.
\end{thm}

\subsection{Preparation for the proof of Theorem~\ref*{thm:CP-hypo}}
\approvals{Behrouz & yes \\ Benoît & yes \\ Stefan & yes}

Before we sketch the proof of Theorem~\ref{thm:CP-hypo} we gather some technical
details in the following lemma, whose proof is contained in the arguments
of~\cite[Sect.~4]{CP13} or those of~\cite[Thm~5.2]{Taji16}
and~\cite[Claim~5.2.1]{Taji16}.
 
\begin{lem}\label{lem:perturb}
  Setting as in Theorem~\ref{thm:CP-hypo}.  If $κ(X,\sL)=\dim X$, then there
  exists a $ℚ$-divisor $B_D ∈ ℚ\Div(X)$, a positive integer $M$, and for
  every $m≥M$ a $ℚ$-divisor $G_m ∈ ℚ\Div(X)$ such that the following holds.
  \begin{enumerate}
  \item The divisor $G_m$ is big.  Its round-down $\lfloor G_m \rfloor$ is zero.
  \item\label{il:xb} There exists a $ℚ$-linear equivalence
    $D + \frac{1}{m}B_D \sim_{ℚ} G_m$.
  \item The divisors $D+\frac{1}{m}B_D$ and $G_m$ have snc support.
  \item\label{il:xd} The pair $(X,G_m)$ is klt.  The pair
    $\bigl(X,D+\frac{1}{m}B_D\bigr)$ is dlt.
  \item The divisor $K_X+D+\frac{1}{m}B_D$ is pseudo-effective.  \qed
  \end{enumerate}
\end{lem}

\subsection{Proof of Theorem~\ref*{thm:CP-hypo}}
\approvals{Behrouz & yes \\ Benoît & yes \\ Stefan & yes}
\CounterStep

Let $M$ be the number and let $B_D$ and $(G_m)_{m ∈ ℕ^+}$ be the divisors of
Lemma~\ref{lem:perturb}.  Choose an ample divisor $A ∈ \Div(X)$, and choose a
divisor $L ∈ \Div(X)$ that represents $\sL$.  Write $n := \dim X$.

\begin{rem}\label{rem:xa}
  As $\sL$ is big, Kodaira's lemma \cite[Prop.~2.2.6]{Laz04-I} implies that for
  any sufficiently positive integer $a$, we have $a·L ≥ A$.
\end{rem}

The next step sets the stage for the proof.

\subsubsection*{Step 1: Minimal models for $(X,G_m)$}
\approvals{Behrouz & yes \\ Benoît & yes \\ Stefan & yes}

The properties listed in Lemma~\ref{lem:perturb} allow us to
use~\cite[Thm.~1.1]{BCHM10} and to conclude that the log-minimal models program
for the pairs $(X,G_m)$ terminate.  In other words, for every number $m$, there
exists a birational map $π_m: X \dashrightarrow X_m'$ consisting of a finite
number of flips and divisorial contractions, such that $(X_m', G_m')$ is klt and
$K_{X_m'}+G_m'$ is nef, where $G_m'$ denotes the strict transform of $G_m$.

We resolve the indeterminacies of $π_m$ by blowing up.  More precisely, choose
log resolutions $μ_m$ of the pair $(X, D+ B_D)$ that fit into a commutative
diagram of birational maps and morphisms as follows,
$$
\xymatrix{
  && \wtilde X_m \ar@/_2mm/[dll]_{μ_m} \ar@/^2mm/[drr]^{\wtilde{π}_m} \\
  X \ar@{-->}[rrrr]_{π_m\text{, MMP for }(X,D_m)} &&&& X_m'.
}
$$
The following will then hold.
\begin{enumerate}
\item The variety $\wtilde X_m$ is smooth.
\item The $μ_m$-exceptional set is of pure codimension one in $\wtilde{X}_m$.
  Let $E_m$ denote the associated reduced divisor.
\item The divisor $\wtilde D_m + \wtilde {B_D}_m + E_m$ has simple normal
  crossing support, where $\wtilde D_m$ and $\wtilde {B_D}_m$ denote the strict
  transforms $D$ and $B_D$, respectively.
\end{enumerate}
Write $\wtilde A_m:=μ_m^*A$, $\wtilde L_m = μ_m^*(L)$ and
$\wtilde \sL_m :=μ_m^*(\sL)$.  With this notation, the standard pull-back of
logarithmic forms then gives an inclusion
$$
\wtilde \sL_m ⊆ \bigotimes^N Ω^1_{\wtilde X_m}\log(\wtilde D_m+ E_m).
$$

\subsubsection*{Step 2: Volume estimates}
\approvals{Behrouz & yes \\ Benoît & yes \\ Stefan & yes}

As a second step in the proof of Theorem~\ref{thm:CP-hypo}, we aim to bound the
volume of $(X_m', G_m')$ from below.  More precisely, the following two claims
will be shown.

\begin{claim}\label{claim:keyineq0}
  Set
  $Δ_m := \wtilde D_m + {\textstyle\frac{1}{m}}\wtilde{B_D}_m+E_m ∈ ℚ\Div
  \bigl(\wtilde{X}_m \bigr)$ and consider the number
  $c := aN·n^{aN-1} ∈ ℕ^+$.  Then, the following inequalities hold for all
  $m ≥ M$, for every ample divisor $H_m ∈ \Div(X_m')$ and every $r∈ ℕ^+$,
  \begin{equation}\label{ineq:first}
    \left[ c· \bigl(K_{\wtilde X_m}+ Δ_m \bigr) - a· \wtilde L_m \right] ·
    \left[ \wtilde{π}_m^*(K_{X_m'}+G_m'+ {\textstyle\frac{1}{r}}H_m) \right]^{n-1} ≥ 0.
  \end{equation}
\end{claim}
\begin{proof}[Proof of Claim~\ref{claim:keyineq0}]
  Pseudo-effectivity of $K_X+D+\frac{1}{m}B_D$ implies that the pull-back
  $μ_m^*(K_X+D+\frac{1}{m}B_D)$ is likewise pseudo-effective, and
  Item~\ref{il:xd} implies that so is $K_{\wtilde X_m}+ Δ_m$.  In particular,
  Theorem~\ref{thm:orbSemPos} (``Generic semipositivity'') applies to the pair
  $(\wtilde X_m, Δ_m)$.  To this end, consider a morphism
  $\wtilde{γ}_m : \wtilde Y → \wtilde X_m$ that is adapted to the divisor
  $Δ_m$.  Recall from \cite[Prop.~4.1.12]{Laz04-I} that we may assume that the
  the cover $\wtilde Y$ is good, in the sense of
  Definition~\ref{defn:goodcover}.  In particular, we may assume that the sheaf
  $Ω^1_{(\wtilde X_m, Δ_m, \wtilde{γ}_m)}$ of adapted differentials is locally
  free.  Next, consider the exact sequence of sheaves
  $$
  0 → \sF_m \longrightarrow \bigotimes^{aN} Ω^1_{(\wtilde X_m, Δ_m,
    \wtilde{γ}_m)} → \sQ_m → 0,
  $$
  where $\sF_m$ is the saturation of $γ_m^* \bigl(\wtilde \sL_m^{⊗ a} \bigr)$
  inside the middle term.  Thanks to Theorem~\ref{thm:orbSemPos} and
  Remark~\ref{rem:polbynef}, the torsion free sheaf $\sQ_m$ verifies the
  inequality:
  $$
  [\sQ_m]·\Bigl[\wtilde{γ}_m^* \Bigl(\underbrace{\wtilde{π}_m^*\bigl(K_{X_m'}
    +G_m' + {\textstyle\frac{1}{r}} H_m\bigr)}_{\text{big and
      nef}}\Bigr)\Bigr]^{n-1}≥ 0.
  $$
  Inequality~\eqref{ineq:first} now follows from an elementary computation of
  the first Chern class of $\sQ_m$, using the equation
  $c_1 \bigl( Ω^1_{(X,Δ,γ)} \bigr) = [γ^*(K_X+ Δ)]$ found in
  Corollary~\ref{cor:det}.  Claim~\ref{claim:keyineq0} is thus shown.
\end{proof}

\begin{claim}\label{claim:keyineq}
  Setting and notation as in Claim~\ref{claim:keyineq0}.  Then, the following
  inequalities hold for all $m ≥ M$, for every ample divisor
  $H_m ∈ \Div(X_m')$ and every $r∈ ℕ^+$,
  \begin{equation}\label{eq:claim-ineq}
    c · \vol \Bigl( K_{X_m'} + G'_m + {\textstyle\frac{1}{r}}H_m \Bigr) ≥
    \bigl[\wtilde A_m\bigr]· \Bigl[ \wtilde{π}_m^* \bigl( K_{X_m'} +G_m' + {\textstyle\frac{1}{r}}H_m \bigr) \Bigr]^{n-1}.
  \end{equation}
\end{claim}
\begin{proof}[Proof of Claim~\ref{claim:keyineq}]
  The birational map $π_m$ is a contraction.  In other words, its inverse
  $π_m^{-1}$ does not contract any divisors.  As a consequence, we see that any
  divisor in $\wtilde{X}_m$ which get contracted by $μ_m$ must also be
  contracted by $\wtilde{π}_m$.  In other words, the $μ_m$-exceptional set
  $E_m ⊆ \wtilde{X}_m$ is also $\wtilde{π}_m$-exceptional.  Together with
  Item~\ref{il:xb}, this observation yields the following $ℚ$-linear
  equivalence,
  \begin{equation}\label{eq:cbbjh}
    K_{\wtilde X_m} + Δ_m \sim_{ℚ} \wtilde{π}_m^* (K_{X_m'}+G'_m) \pm (\wtilde{π}_m\text{-exceptional}).
  \end{equation}
  Claim~\ref{claim:keyineq} now follows by putting things together.
  \begin{align*}
    \mathrlap{\bigl[\wtilde A_m\bigr]· \Bigl[ \wtilde{π}_m^* \bigl( K_{X_m'} +G_m' + {\textstyle\frac{1}{r}}H_m \bigr) \Bigr]^{n-1}}\qquad \\
 & ≤ a·\bigl[\wtilde L_m\bigr]· \Bigl[ \wtilde{π}_m^* \bigl( K_{X_m'} +G_m' + {\textstyle\frac{1}{r}}H_m \bigr) \Bigr]^{n-1} && \text{Remark~\ref{rem:xa}}\\
 & ≤ c·\bigl[K_{\wtilde X_m}+Δ_m\bigr]· \Bigl[ \wtilde{π}_m^* \bigl( K_{X_m'} +G_m' + {\textstyle\frac{1}{r}}H_m \bigr) \Bigr]^{n-1} && \text{Inequality~\eqref{ineq:first}}\\
 & ≤ c·\bigl[\wtilde{π}_m^* \bigl( K_{X_m'} +G_m' \bigr) \bigr]· \Bigl[ \wtilde{π}_m^* \bigl( K_{X_m'} +G_m' + {\textstyle\frac{1}{r}}H_m \bigr) \Bigr]^{n-1} && \text{Equation~\eqref{eq:cbbjh}} \\
 & ≤ c·\vol \Bigl( K_{X_m'} + G'_m + {\textstyle\frac{1}{r}}H_m \Bigr)
  \end{align*}
  Claim~\ref{claim:keyineq} is thus established.
\end{proof}

\subsubsection*{Step 3: Application of Claim~\ref{claim:keyineq}, end of proof}
\approvals{Behrouz & yes \\ Benoît & yes \\ Stefan & yes}

According to Teissier inequality for nef divisors, \cite[Thm.~1.6.1]{Laz04-I},
Inequality~\eqref{eq:claim-ineq} implies that
$$
c·\vol \Bigl( K_{X_m'} +G_m' + {\textstyle\frac{1}{r}}H_m \Bigr) ≥
\vol \Bigl(\wtilde A_m \Bigr)^{1/n} · \vol \Bigl(K_{X_m'}+G_m'+ {\textstyle\frac{1}{r}}H_m \Bigr)^{(n-1)/n}.
$$
In other words,
$$
\vol \Bigl( K_{X_m'}+G_m' + \frac{1}{r}H_m \Bigr) ≥ c^{-n} · \vol \bigl(\wtilde A_m \bigr).
$$
By taking $r→ \infty$, we find that
$$
\vol \bigl(K_{X_m'}+G_m' \bigr)≥ c^{-n}·\vol(\wtilde A_m).
$$
On the other hand, we know thanks to the negativity lemma in the minimal models
program, that $\vol(K_{X_m'}+G_m') = \vol(K_X+G_m)$ and as
$\vol(\wtilde A_m) = \vol(A)$, we have
$$
\vol \Bigl(K_X+D+ {\textstyle\frac{1}{m}}B_D\Bigr) ≥ c^{-n}·\vol(A).
$$
Theorem~\ref{thm:CP-hypo} now follows by taking $m→ \infty$.  \qed

\part{Proof of the semipositivity result}
%
% Do not edit the following line.  The text is automatically updated by
% subversion.
%
\svnid{$Id: 07-positivity.tex 149 2016-08-19 06:52:03Z kebekus $}

\section{Positivity of relative dualising sheaves}
\subversionInfo
\approvals{Behrouz & yes \\ Benoît & yes \\ Stefan & yes}

As we shall see in Section~\ref{sect:failure}, the orbifold generic
semipositivity result, Theorem~\ref{thm:orbSemPos}, is proved by contradiction.
More precisely, given a pair $(X,Δ)$ with pseudo-effective $K_X+Δ$, and after
integrability considerations (Subsection~\ref{subsect:integrability}), the
existence of a subsheaf of $\sT_{(X,γ,Δ)}$ with positive slope leads to an
algebraic foliation on $X$.  The negativity properties of the relative canonical
sheaf of the rational map associated to this foliation brings about the required
contradiction to the pseudo-effectivity assumption of $K_X+Δ$.

\begin{thm}[\protect{\cite[Thm.~2.11]{CP13}}]
\label{thm:neat}
  Let $f : X \dasharrow Z$ be a rational map with connected fibres between
  normal, projective varieties.  Assume that $f$ is essentially
  equidimensional\footnote{See Definition~\vref{defn:esseqdim}.}, that $X$ is
  $ℚ$-factorial, that there exists a $ℚ$-Weil divisor $Δ$ on $X$ such that
  $(X,Δ)$ is log canonical, and that $K_X+Δ$ is pseudo-effective.  If
  $(C_t)_{t ∈ T}$ is a family of curves that dominates $X$ and avoids small
  sets\footnote{See Notation~\vref{not:curves}.}, then
  $$
  \bigl[ K_{X/Z}+Δ^{horiz}-\Ramification f \bigr]·[C_t] ≥ 0,\quad
  \text{for all $t ∈ T$}.
  $$
\end{thm}

\begin{rem}[$ℚ$-factoriality in Theorem~\ref*{thm:neat}]
  The assumption that $X$ is $ℚ$-factorial is posed for notational convenience.
  It implies that the intersection numbers in the displayed formula are
  well-defined.  Note, however, that almost all curves in the family
  $(C_t)_{t ∈ T}$ stay away from the singularities of $X$.  Restricting to these
  curves, and replacing $X$ with a suitable log-resolution, it is possible to
  obtain a more general result at the cost of additional and more complicated
  notation.
\end{rem}

Theorem~\ref{thm:neat} is a consequence of positivity results for direct images
of relative dualising sheaves, which we present in Theorem~\ref{thm:pdis} in the
form of a pseudo-effectivity result for the relative dualising sheaf.  These
results have a long history, cf.~\cite{MR2674856} for a survey in the setting
where $Δ = 0$.  In case where the coefficients of $Δ$ are of the form
$\frac{m-1}{m}$, the relevant positivity results for the push-forward of
$\sO_X \bigl( m·(K_{X/Z}+Δ) \bigr)$ appear in \cite[Sect.~9]{Lu02} and
\cite[Thm.~4.11]{Cam04}; the paper \cite{MR2449950} approaches the case where
$Δ$ is integral by analytic methods.  The general case, where the coefficients
of $Δ$ are arbitrary, is treated in \cite[Thm.~1.1]{Fujino14} and (again in the
analytic setting) in \cite[Cor.~5.2.1]{PT14}.  For notational convenience, we
choose \cite{Fujino14} as our main reference.

\begin{thm}[Pseudo-effectivity of relative dualising sheaves]\label{thm:pdis}
  Let $f : X → Z$ be a surjective morphism with connected fibres between smooth,
  projective varieties.  Let $Δ$ be a $ℚ$-divisor with snc support on $X$ such
  that $(X,Δ)$ is log canonical.  Assume further that the restriction of
  $K_{X/Z}+Δ$ to the general $f$-fibre is is pseudo-effective.  Then $K_{X/Z}+Δ$
  is pseudo-effective.
\end{thm}
\begin{proof}
  Choose a very ample prime divisor $H$ on $X$ that is general in its linear
  systems.  For every sufficiently small rational number $0 < ε \ll 1$, the
  divisor $Δ+ε·H$ will then have snc support, the pair $(X,Δ+ε·H)$ will be log
  canonical and if $F ⊆ X$ is any general $f$-fibre, then $(K_{X/Z}+Δ+ε·H)|_F$
  will be big.  Choose a number $m \gg 0$ such that the divisor
  $m·(K_{X/Z}+Δ+ε·H)$ is integral and such that
  $$
  \sE := f_* \sO_X \bigl( m·(K_{X/Z}+Δ+ε·H) \bigr)
  $$
  is of positive rank.  The twisted weak positivity theorem of
  \cite[Thm.~1.1]{Fujino14} asserts that $\sE$ is weakly positive.  In other
  words, cf.~\cite[Sect.~7]{Fujino14}, there exists a dense, Zariski-open set
  $Y° ⊆ Y$ such that $\sE$ is weakly positive over $U$, in the sense of Viehweg,
  \cite[Defn.~2.11]{Viehweg95}.  Its pull-back $f^* \sE$ is then likewise weakly
  positive, \cite[Lem.~2.15]{Viehweg95}, and so is the target of the natural,
  non-trivial morphism
  $$
  f^* \sE = f^* f_* \sO_X \bigl( m·(K_{X/Z}+Δ+ε·H) \bigr) → \sO_X \bigl(
  m·(K_{X/Z}+Δ+ε·H) \bigr),
  $$
  cf.~\cite[Lem.~2.16]{Viehweg95}.  Since the target is invertible, it follows
  immediately from the definition of ``weakly positive'' that the $ℚ$-divisor
  $K_{X/Z}+Δ+ε·H$ is pseudo-effective, \cite[Rem.~7.6]{Fujino14}.  Conclude by
  taking the limit $ε → 0$.
\end{proof}

\subsection{Proof of Theorem~\ref*{thm:neat}}
\approvals{Behrouz & yes \\ Benoît & yes \\ Stefan & yes}

The proof of Theorem~\ref{thm:neat} essentially consists of two parts.  Part one
is focused on modifying the rational map $f: X\dasharrow Z$ and its subsequent
replacement by a morphism that fits the premise of Theorem~\ref{thm:pdis}.  This
is roughly the content of Step.~1--3 and finally Step.~4 (see
Consequences.~\ref{cons:psef3} and~\ref{cons:psef4}).  In Step.~5 it then
becomes evident that Theorem~\ref{thm:neat} is an immediate consequence of
Theorem~\ref{thm:pdis}.

\subsubsection*{Step 1: Resolution and base change}
\approvals{Behrouz & yes \\ Benoît & yes \\ Stefan & yes}

Following the construction steps outlined below, we construct a commutative
diagram of morphisms and maps as follows,
$$
\xymatrix{ %
  \wtilde{X} \ar[rrrrrr]^a_{\text{log resolution of fibre product}} \ar[d]_{\wtilde{f}} &&&&&& \bar{X} \ar[rrr]^b_{\text{resol.\ of $f$ and }(X,Δ)} \ar[d]_{\bar{f}} &&& X \ar@{-->}[d]^f \\
  \wtilde{Z} \ar[rrr]^{α}_{\text{strong log resol.}} &&& \what{Z} \ar[rrr]^{β}_{\substack{\text{strongly adpt.\ cover for}\\(Z,\, \OrbiBranch(f))}} &&& Z \ar@{=}[rrr] &&& Z.
}
$$
\begin{enumerate}
\item Choose a strong log resolution of the morphism $f$ and the pair $(X,Δ)$.
  We obtain a smooth variety $\bar{X}$ and birational morphism $b: \bar{X} → X$
  from a normal variety that is isomorphic over $\Defn(f) ∩ (X,Δ)_{\snc}$, such
  that $\bar{f} := f ◦ b$ is a morphism and such that the $b$-exceptional locus
  $E^{b}$ as well as $E^{b}+ b_*^{-1} Δ$ are divisors with simple normal
  crossing support.
  
\item Consider the divisor $\OrbiBranch(f)$ that was introduced in
  Definition~\vref{defn:RessEq}.  Proposition~\ref{prop:exsam} allows to choose
  a strongly adapted cover $β : \what{Z} → Z$ that is associated with the pair
  $\bigl( Z,\, \OrbiBranch(f)\bigr)$.  Since $β$ is finite,
  Construction~\ref{cons:pull} allows to consider the pull-back divisor
  $β^* \OrbiBranch (f)$.

\item Choose a strong log resolution of the pair
  $\bigl(\what{Z},\, β^* \OrbiBranch (f)\bigr)$.  We obtain a smooth variety
  $\wtilde{Z}$ and a birational morphism $α : \wtilde{Z} → \what{Z}$ that is
  isomorphic wherever $\bigl(\what{Z},\, β^* \OrbiBranch(f)\bigr)$ is snc.  Both
  the $α$-exceptional locus $E^{α}$ as well as
  $E^{α}+ α_*^{-1} \bigl(β^* \OrbiBranch (f) \bigr)$ are divisors with simple
  normal crossing support.
  
\item\label{cons:X4} Choose a strong log resolution of the fibre product
  $\wtilde{Z} ⨯_Z \bar{X}$.  We obtain a smooth variety $\wtilde{X}$.  Composed
  with the projection to the second factor, the resolution yields a generically
  finite morphism $a : \wtilde{X} → \bar{X}$.  Let $E^{b ◦ a} ⊆ \wtilde{X}$ be
  the union of those divisors that are contracted by $b ◦ a$.  Then, both
  $E^{b ◦ a}$ as well as $E^{b ◦ a} + a^* (b_*^{-1} Δ)$ are divisors with simple
  normal crossing support.
\end{enumerate}

\subsubsection*{Step 2: Open sets and local normal forms}
\approvals{Behrouz & yes \\ Benoît & yes \\ Stefan & yes}

\CounterStep Let $Z° ⊆ Z$ and $X° ⊆ f^{-1}(Z°)$ be the maximal open sets such
that the following holds.

\begin{enumerate}
\item The pairs $(X°,Δ+ \Ramification f)$ and $(Z°, \Branch f)$ are both snc.
\item The morphism $f|_{X°}$ is equidimensional and can locally be written in
  normal form, in the sense of Notation~\ref{not:localNormalForm}.
\item\label{il:q-4} Setting $\wtilde{Z}° := (β◦ α)^{-1}(Z°)$, the morphism
  $(β ◦ α)|_{\wtilde{Z}°} : \wtilde{Z}° → Z°$ is finite and can locally be
  written in normal form.
\item\label{il:q-5} Setting $\bar{X}° := b^{-1}(X°)$, the morphism
  $b° := b|_{\bar{X}°} : \bar{X}° → X°$ is isomorphic.  In particular,
  $\bar{f}|_{\bar{X}°} : \bar{X}° → Z°$ can locally be written in normal form.
\item\label{il:q-6} Setting $\wtilde{X}° := a^{-1}(\bar{X}°)$, the morphism
  $a° := a|_{\wtilde{X}°} : \wtilde{X}° → \bar{X}°$ is finite and can locally be
  written in normal form.
\end{enumerate}

\begin{obs}\label{obs:D0}
  Recall from Construction~\ref{cons:E:8} and from Zariski's main theorem that
  $Z°$ and $X°$ are big open sets of $Z$ and $X$, respectively.
\end{obs}

\subsubsection*{Step 3: Adjunction for the morphism $b$}
\approvals{Behrouz & yes \\ Benoît & yes \\ Stefan & yes}

Decompose the $b$-exceptional divisor $E^b$ into irreducible components,
$(E^b_i)_{i ∈ I_b}$.  Since $(X,Δ)$ is lc, the standard adjunction for the
morphism $b$ reads
$$
K_{\bar{X}} + b_*^{-1} Δ = b^* \bigl( K_X + Δ \bigr) + \sum a_i E_i^b, \quad
\text{with all } a_i ≥ -1.
$$
The $ℚ$-divisor
$$
\bar{Δ} := b_*^{-1} Δ - \sum_{a_i < 0} a_i E_i^b
$$
is effective with coefficients from the interval $[0,1] ∩ ℚ$, and has simple
normal crossings support.  The pair $(\bar{X}, \bar{Δ})$ is thus log canonical,
and
$$
K_{\bar{X}} + \bar{Δ} = b^* \bigl( \underbrace{K_X + Δ}_{\mathclap{\text{psef by
      assumpt.}}} \bigr) + \sum_{a_i > 0} a_i E_i^b
$$
is again pseudo-effective.  To end with Step~3, set
$\bar{Δ}^h := \bar{Δ} - b_*^{-1} Δ^{vert}$ and observe that if
$\bar{F} ⊆ \bar{X}$ is a general $\bar{f}$-fibre, then $\bar{F}$ is disjoint
from the support of $b_*^{-1} Δ^{vert}$.  The following is thus an immediate
consequence.

\begin{obs}\label{obs:psef2}
  We have $\bar{Δ}^h|_{\bar{X}°} = (b°)^* Δ^{horiz}$.  The pair
  $(\bar{X}, \bar{Δ}^h)$ is log-canonical and the restricted divisor
  $\bigl( K_{\bar{X}} + \bar{Δ}^h \bigr)|_{\bar{F}}$ is pseudo-effective.  \qed
\end{obs}

\subsubsection*{Step 4: Adjunction for the morphism $a$}
\approvals{Behrouz & yes \\ Benoît & yes \\ Stefan & yes}

\CounterStep Using Items~\ref{il:q-4} and \ref{il:q-5}, the local normal form of
$(β ◦ α)|_{\wtilde{Z}°}$ and $\bar{f}|_{\bar{X}°}$, as well as the construction
of $β$ as a strongly adapted cover, an elementary computation in local
coordinates gives the following $ℚ$-linear equivalence,
$$
K_{\wtilde{X}°/\wtilde{Z}°} \sim_{ℚ} (b°◦a°)^* \bigl( K_{X°/Z°} -
\Ramification(f) \bigr).
$$
A similar equation holds for pairs.  For a precise formulation, let
$\wtilde{ι}: \wtilde{X}° → \wtilde{X}$ be the obvious inclusion map and set
$$
\wtilde{Δ}^h := \wtilde{ι}_* \Bigl( (a°)^*(\bar{Δ}^h) \Bigr),
$$
where the push-forward $\wtilde{ι}_*$ is taken in the sense of
Construction~\ref{cons:push} and Remark~\ref{rem:push}.  It follows from the
construction of the morphism $a$ in \ref{cons:X4} that every component of
$\wtilde{Δ}^h$ dominates $\wtilde{Z}$, and that no component of
$\wtilde{Δ}^h|_{\wtilde{X}°}$ is contained in the ramification locus of the
finite morphism $a°$.  Likewise, if $\wtilde{F} ⊆ \wtilde{X}$ is any general
fibre of $\wtilde{f}$, it follows from construction that $a$ is étale near
$\wtilde{F}$.  The following are thus immediate consequences of
Observation~\ref{obs:psef2}.

\begin{consequence}\label{cons:psef3}
  The pair $(\wtilde{X}, \wtilde{Δ}^h)$ is log-canonical and satisfies
  $$
  \bigl( K_{\wtilde{X}/\wtilde{Z}} + \wtilde{Δ}^h\bigr)|_{\wtilde{X}°} \sim_{ℚ}
  (b°◦a°)^* \bigl( K_{X°/Z°} - \Ramification(f) \bigr).  \qed
  $$
\end{consequence}

\begin{consequence}\label{cons:psef4}
  The restricted divisor
  $\bigl( K_{\wtilde{X}} + \wtilde{Δ}^h \bigr)|_{\wtilde{F}}$ is
  pseudo-effective.  \qed
\end{consequence}

\subsubsection*{Step 5: End of proof}
\approvals{Behrouz & yes \\ Benoît & yes \\ Stefan & yes}

The family $(C_t)_{t ∈ T}$ avoids small sets by assumption.  Together with
Observation~\ref{obs:D0}, this means that if $t ∈ T$ is general, then the curve
$C_t$ is contained in $X°$.  By \ref{il:q-6}, its preimage
$\wtilde{C}_t := (b ◦ a)^{-1}(C_t)$ is a curve in $\wtilde{X}°$ whose components
are movable curves in $\wtilde{X}$.  Since the cycle-theoretic push-forward
$(b ◦ a)_* (\wtilde{C}_t)$ is a positive multiple of $C_t$, we obtain
\begin{align*}
  0 & ≤ \bigl( K_{\wtilde{X}/\wtilde{Z}} + \wtilde{Δ}^h \bigr)·\wtilde{C}_t && \text{Thm.~\ref{thm:pdis} and Cons.~\ref{cons:psef4}} \\
    & = (b ◦ a)^* \bigl( K_{X/Z}+Δ^{hor}-\Ramification(f) \bigr)·\wtilde{C}_t && \text{Cons.~\ref{cons:psef3}} \\
    & = \const^+·\bigl( K_{X/Z}+Δ^{hor}-\Ramification(f) \bigr)·C_t && \text{Proj.\ formula.}
\end{align*}
This finishes the proof of Theorem~\ref{thm:neat}.  \qed

%
% Do not edit the following line.  The text is automatically updated by
% subversion.
%
\svnid{$Id: 08-integrability.tex 162 2016-08-23 23:43:13Z taji $}

\section{Failure of semipositivity, construction of morphisms}
\label{sect:failure}
\subversionInfo
\approvals{Behrouz & yes \\ Benoît & yes \\ Stefan & yes}

As was mentioned earlier, a key component of Campana-Păun's proof of the generic
semipositivity result is the observation that given a lc pair $(X,D)$ and an
adapted cover $γ:Y → X$, any subsheaf $\sF_{(X,Δ,γ)} ⊆ \sT_{(X,Δ,γ)}$ that is
maximally destabilising respect to $γ^*(\text{ample})$ induces an algebraic
foliation on $X$, whose leaves are often algebraic.  This is the content of the
next theorem.

\begin{thm}[\protect{From positive subsheaves to foliations, cf.~\cite[Sect.~2]{CP13}}]\label{thm:exMor}
  Let $(X,Δ)$ be a projective pair and let $γ : Y → X$ be an adapted morphism
  that is Galois with group $G$.  Assume that $Ω^{[1]}_{(X,Δ,γ)}$ is \emph{not}
  $γ$-generically semipositive.  Then, there exists a normal variety $Z$, a
  dominant, essentially equidimensional rational map $ψ: X \dasharrow Z$, and a
  family $(C_t)_{t ∈ T}$ of curves in $X$ such that the following holds.
  \begin{enumerate}
  \item\label{il:a1} The family $(C_t)_{t ∈ T}$ dominates $X$ and avoids small
    sets.
  \item\label{il:a2} Given any $t ∈ T$, we have
    $[\sT_{X/Z}]·[C_t] > [Δ^{horiz}]·[C_t]$.
  \end{enumerate}
\end{thm}

\begin{rem}
  The family $(C_t)_{t ∈ T}$ avoids small sets, and its general members are thus
  contained in $X_{\reg}$.  The intersection numbers of Item~\ref{il:a1} are
  therefore well-defined, even if $X$ is not necessarily $ℚ$-factorial.
\end{rem}

\subsection{Proof of Theorem~\ref*{thm:exMor}}
\label{subsect:integrability}
\approvals{Behrouz & yes \\ Benoît & yes \\ Stefan & yes}

As before, the proof is subdivided into a number of relatively independent
steps.

\subsubsection*{Step 1: Setup}
\approvals{Behrouz & yes \\ Benoît & yes \\ Stefan & yes}

By assumption, there exists a very ample divisor $A$ on $X$ such that
$Ω^{[1]}_{(X,Δ,γ)}$ is \emph{not} generically semipositive with respect to the
ample divisor $A_γ := γ^* A$.  Let and $C_γ ⊂ Y$ be a a general complete
intersection curve for the ample divisor $A_Y$ on $Y$, in the sense of
Mehta-Ramanathan.  Consider the Harder-Narasimhan filtration of the sheaf
$\sT_{(X,Δ,γ)}$ of adapted tangents with respect to $A_γ$ and let
$$
\sF_{(X,Δ,γ)} ⊆ \sT_{(X,Δ,γ)}
$$
denote the maximally destabilising subsheaf, which is saturated in
$\sT_{(X,Δ,γ)}$ and hence reflexive.  By assumption, its slope is positive,
$μ_{A_γ} \bigl( \sF_{(X,Δ,γ)} \bigr) > 0$.  Remark~\ref{rem:X2} yields an
inclusion
$$
\sF_{(X,Δ,γ)} ⊆ γ^{[*]} \sT_X(- \log \lfloor Δ \rfloor).
$$
We denote the associated saturation by
$\sF_{(X,Δ,γ)}^{\sat} ⊆ γ^{[*]} \sT_X(- \log \lfloor Δ \rfloor)$.  Since
$\sF_{(X,Δ,γ)} ⊆ \sT_{(X,Δ,γ)}$ is itself saturated, we have an equality
$$
\sF_{(X,Δ,γ)} = \sF_{(X,Δ,γ)}^{\sat} ∩ \sT_{(X,Δ,γ)}.
$$

\begin{obs}[Regularity along $C_γ$]\label{obs:o1}
  The curve $C_γ$ is a general member in a dominating family of curves that
  avoids small sets.  In particular, the following holds.
  \begin{enumerate}
  \item The curve $C_γ$ is entirely contained in the smooth locus of $Y$ and is
    itself smooth.
  \item The sheaves $\sF_{(X,Δ,γ)}|_{C_γ}$ and $\sF_{(X,Δ,γ)}^{\sat}|_{C_γ}$ are
    locally free.
  \end{enumerate}
\end{obs}

\begin{obs}[Positivity along $C_γ$]\label{obs:o1a}
  The locally free sheaf $\sF_{(X,Δ,γ)}|_{C_γ}$ is semistable, of positive
  degree and therefore ample.  The larger sheaf $\sF_{(X,Δ,γ)}^{\sat}|_{C_γ}$ is
  likewise ample.
\end{obs}

\begin{obs}[$G$-invariance]\label{obs:o2}
  The divisor $A_γ$ is invariant under the action of the Galois group.  As a
  consequence, it follows from the uniqueness of the Harder-Narasimhan
  filtration that the maximally destabilising subsheaf $\sF_{(X,Δ,γ)}$ is a
  $G$-subsheaf of $\sT_{(X,Δ,γ)}$, and also of
  $γ^{[*]} \sT_X(- \log \lfloor Δ \rfloor)$.

  In a similar vein, it follows from uniqueness of saturation that
  $\sF_{(X,Δ,γ)}^{\sat}$ is a $G$-subsheaf of
  $γ^{[*]} \sT_X(- \log \lfloor Δ \rfloor)$.  Thus,
  by~\cite[Prop.~2.16]{GKPT15}, there exists a reflexive, saturated subsheaf
  $\sF ⊆ \sT_X(-\log \lfloor Δ \rfloor)$ such that
  $\sF_{(X,Δ,γ)}^{\sat} = γ^{[*]} \sF$.
\end{obs}

\begin{notation}
  Let $\sF^{\sat} ⊆ \sT_X$ denote the saturation of $\sF$ inside $\sT_X$.
  The following diagrams summarise the situation,
  $$
  \sT_Y \bigl( -\log Δ_γ \bigr) \quad ⊆ \quad
  \underbrace{\sT_{(X,Δ,γ)}}_{\mathclap{\text{contains }\sF_{(X,Δ,γ)}}} \quad ⊆
  \quad \underbrace{γ^{[*]} \sT_X(-\log \lfloor Δ \rfloor)}_{\text{contains
    }\sF_{(X,Δ,γ)}^{\sat} = γ^{[*]} \sF}
  $$
  and
  $$
  \underbrace{\sT_X(- \log \lfloor Δ \rfloor) }_{\mathclap{\text{contains
      }\sF}} \quad ⊆ \quad \underbrace{\sT_X}_{\mathclap{\text{contains }\sF^{\sat}}}.
  $$
\end{notation}

\begin{obs}[Regularity and amplitude along $C$]\label{obs:o3}
  Since $γ$ is finite, the curve $C := γ(C_γ)$ is again a general member in a
  dominating family of curves that avoids small sets.  It follows that $C$ is
  contained in the smooth locus of $X$ and that $\sF$ is locally free near $C$.
  In particular,
  $$
  \sF_{(X,Δ,γ)}^{\sat}|_{C_γ} = (γ|_{C_γ})^* (\sF|_C),
  $$
  and \cite[Prop.~6.1.8]{Laz04-II} therefore implies that $\sF|_C$ and
  $\sF^{\sat}|_C$ are both ample.
\end{obs}

\subsubsection*{Step 2: Construction of a foliation}
\approvals{Behrouz & yes \\ Benoît & yes \\ Stefan & yes}

Next, we will show that the sheaf $\sF^{\sat}$ is in fact a foliation.  For
this, Proposition~\ref{prop:sddffg}, which describes the lifting the O'Neil
tensor to an adapted cover, will be the key ingredient.

\begin{claim}\label{claim:P2}
  The sheaf $\sF^{\sat} ⊆ \sT_X$ is closed under the Lie-bracket.
\end{claim}
\begin{proof}[Proof of Claim~\ref{claim:P2}]
  Closedness under Lie bracket can be checked on an open subset.  It will
  therefore suffice to show that the subsheaf
  $\sF ⊆ \sT_X(-\log \lfloor Δ \rfloor)$, which agrees with $\sF^{\sat}$
  generically, is closed under the Lie-bracket.  Equivalently, we need to prove
  that the O'Neil tensor
  $$
  N : \sF^{[2]} → \bigl(\factor{\sT_X(-\log \lfloor Δ \rfloor)}{\sF} \bigr)^{**}
  $$
  vanishes.  For this, recall from Proposition~\ref{prop:sddffg} that the
  restriction of its reflexive pull-back to $\sF_{(X,Δ,γ)}$ gives a morphism
  $$
  N_{(X,Δ,γ)} : \sF_{(X,Δ,γ)}^{[2]} →
  \bigl(\factor{\sT_{(X,Δ,γ)}}{\sF_{(X,Δ,γ)}} \bigr)^{**}.
  $$
  Since $γ^{[*]} \sF$ and $\sF_{(X,Δ,γ)}$ agree on a dense open set, it will be
  enough to show that $N_{(X,Δ,γ)}$ vanishes.  This is actually the case for
  slope reasons.  We have the following inequalities:
  \begin{align*}
    μ_{A_{γ}}^{\min} \bigl(\sF_{(X,Δ,γ)}^{[2]} \bigr) & = 2·μ_{A_{γ}}^{\min} \bigl(\sF_{(X,Δ,γ)} \bigr) \\
                                                      & > μ_{A_{γ}}^{\min}(\sF_{(X,Δ,γ)}) && \text{since $μ^{\min} > 0$ by assumption} \\
                                                      & ≥ μ_{A_{γ}}^{\max} \Bigl(\factor{\sT_{(X,Δ,γ)}}{\sF_{(X,Δ,γ)}} \Bigr) && \text{since $\sF_{(X,Δ,γ)}$ is max.\ destab.}
  \end{align*}
  Claim~\ref{claim:P2} thus follows.
\end{proof}

\begin{notation}
  Following Notation~\ref{not:genTrans} the foliation $\sF^{\sat}$ induces a
  decomposition $Δ = Δ^{trans}+Δ^{ntrans}$.
\end{notation}

\subsubsection*{Step 3: Construction of a morphism}
\approvals{Behrouz & yes \\ Benoît & yes \\ Stefan & yes}

We show that the foliation $\sF^{\sat}$ is algebraic and therefore defines an
essentially equidimensional rational map.  The algebraicity criterion that we
employ goes back to Hartshorne, \cite[Thm.~6.7]{Ha68}.  We refer the
reader to \cite{KST07} or to one of the papers \cite{Bost01, BMcQ01, BM16} for a
thorough discussion.

\begin{claim}\label{claim:gianluca}
  There exists a normal, projective variety $Z$, and a dominant, essentially
  equidimensional, rational map $ψ : X \dasharrow Z$ such that the sheaves
  $\sT_{X/Z}$ and $\sF^{\sat}$ agree.
\end{claim}
\begin{proof}[Proof of Claim~\ref{claim:gianluca}]
  In the setting at hand, amplitude of the restricted foliation $\sF|_C$ implies
  that any leaf of $\sF$ that intersects $C$ is automatically algebraic; we
  refer the reader to \cite[Thm.~1]{KST07} for a convenient reference.  Since
  $C$ is general in a dominating family, this gives rise to a rational map
  $$
  φ : X \dasharrow X ⨯ \Chow(X), \quad x \mapsto \left( x, \bigl[
    \overline{\text{leaf through $x$}} \bigr] \right).
  $$
  Since $X$ is normal, it follows from Zariski's main theorem,
  \cite[V~Thm.~5.2]{Ha77}, that there exists a big open set $U ⊆ X$ where $φ$ is
  well-defined.  Observe that the following holds:
  \begin{itemize}
  \item The morphism $φ|_U$ has a right inverse and is therefore necessarily
    injective, in particular quasi-finite.
  \item The image of the morphism $φ$ is contained in the universal family that
    exists for the Chow variety.  Let $\wtilde{V}$ be the normalisation of the
    closure of the image.
  \end{itemize}

  The universal family over $\Chow(X)$ need not be normal.  Normalising, we
  obtain a diagram as follows,
  $$
  \xymatrix{ %
    U \ar[r] \ar@/^4mm/[rr]^{\text{well-defined, injective}} & X
    \ar@{-->}[r]_{φ} & \widetilde{V} \ar[rr]_{\text{normalisation}}
    \ar@/^4mm/[rrrr]^{\text{equidim.}} && \Univ \ar[rr]_{\text{equidim.}} &&
    \Chow(X).  %
  }
  $$
  More is true.  Zariski's Main Theorem in the form of
  Grothendieck\footnote{Zariski's Main Theorem in the form of Grothendieck is
    found in \cite{EGA4-3}.  We refer the reader to \cite[Sect.~3.4]{GKP13} for
    a more detailed discussion.}, asserts that the map $U → \widetilde{V}$ is in
  fact an open immersion.  In summary, we see that the composed morphism
  $U → \Chow(X)$ is equidimensional.  Let $Z$ be the normalisation of the image,
  and $ψ : X \dasharrow Z$ the induced map.

  The sheaves $\sT_{X/Z}$ and $\sF$ agree on an open subset of $U$, where all
  spaces and foliations are regular, and all maps are well-defined and smooth.
  Since both are saturated subsheaves of $\sT_X$, this suffices to show that
  they agree everywhere.  Claim~\ref{claim:gianluca} follows.
\end{proof}

\begin{notation}
  Following Notation~\ref{not:decomphv}, the map $ψ : X \dasharrow Z$ induces a
  decomposition $Δ = Δ^{horiz}+Δ^{vert}$.
\end{notation}

\begin{obs}
  The decomposition of $Δ$ agrees with that coming from the foliation.  In other
  words, $Δ^{trans} = Δ^{horiz}$ and $Δ^{ntrans} = Δ^{vert}$.
\end{obs}

\subsubsection*{Step 4: End of proof}
\approvals{Behrouz & yes \\ Benoît & yes \\ Stefan & yes}

To end the proof, we need to show that the rational map $ψ$ satisfies
Inequality~\ref{il:a2}.  We aim to apply Proposition~\ref{prop:xad}.  To this
end, recall from our construction that
$$
\sF = \sF^{\sat} ∩ \sT_X(-\log \lfloor Δ \rfloor) \quad\text{and}\quad
\sF_{(X,Δ,γ)} = γ^{[*]}\sF ∩ \sT_{(X,Δ,γ)}.
$$
Item~\eqref{il:Y2} of Proposition~\ref{prop:xad} thus gives an equality of
intersection numbers
\begin{align*}
C_γ·[γ^{[*]}\sT_{X/Z}] & = \underbrace{C_γ · [\sF_{(X,Δ,γ)}]}_{> 0\text{ by Obs.~\ref{obs:o1a}}} +\; C_γ·[γ^* Δ^{trans}] + \underbrace{C_γ·[\text{effective}]}_{≥ 0\text{ since $C_γ$ movable}} \\
  & > C_γ·[γ^* Δ^{trans}].
\end{align*}
In summary, we see that $C_γ·[γ^{[*]}\sT_{X/Z}] > C_γ·[γ^* Δ^{trans}]$ and hence
$C·[\sT_{X/Z}] > C·[Δ^{trans}]$ as claimed.  This finishes the proof of
Theorem~\ref{thm:exMor}.  \qed

%
% Do not edit the following line.  The text is automatically updated by
% subversion.
%
\svnid{$Id: 09-genericSemipositivity.tex 162 2016-08-23 23:43:13Z taji $}

\section{Proof of the semipositivity result}
\label{ssec:potosp}
\subversionInfo
\approvals{Behrouz & yes \\ Benoît & yes \\ Stefan & yes}

We prove Theorem~\ref{thm:orbSemPos} in this section.  With the preparations at
hand, the proof is now quite short.  We argue by contradiction and assume that
$Ω^{[1]}_{(X,Δ,γ)}$ is \emph{not} $γ$-generically semipositive.  As we have seen
in Theorem~\ref{thm:exMor}, this implies the existence of a normal variety $Z$,
a dominant, essentially equidimensional, rational map $f: X \dasharrow Z$, and a
family $(C_t)_{t ∈ T}$ of curves that dominates $X$ and avoids small sets, such
that the following inequality holds for all $t ∈ T$,
$$
[\sT_{X/Z}]·[C_t] > [Δ^{horiz}]·[C_t].
$$
Recalling the description of $\sT_{X/Z}$ given in Lemma~\ref{lem:relTx}, this is
equivalent to
$$
[K_{X/Z}+ Δ^{horiz} - \Ramification f]·[C_t] < 0,
$$
contradicting the positivity of relative dualising sheaves that was established
in Theorem~\ref{thm:neat}, and ending the proof of Theorem~\ref{thm:orbSemPos}.
\qed

\end{document}